\documentclass[12pt]{amsart}

\usepackage{amsmath,amssymb,amsfonts,amsthm,amscd,indentfirst}

\usepackage{amsmath,amssymb,amsfonts,amsthm,amscd}
\usepackage{indentfirst}
\usepackage{hyperref}

\usepackage{color}

\numberwithin{equation}{section}

\newtheorem{prop}{Proposition}[section]
\newtheorem{theo}[prop]{Theorem}
\newtheorem{lemm}[prop]{Lemma}
\newtheorem{coro}[prop]{Corollary}
\newtheorem{defi}[prop]{Definition}

\theoremstyle{remark}
\newtheorem{exam}{Example}
\newtheorem{rema}{Remark}[section]

\def\and{\quad{\rm and}\quad}

\def\<{\langle}
\def\>{\rangle}

\def\p{\partial}
\def\be{\begin{equation}}
\def\ee{\end{equation}}

\usepackage{graphicx}

\newcounter{mnotecount}[section]

\def\R{\mathbb{R}}
\def\SM{\mathbb{M}^{n+1}_m}
\def\Ric{\mathrm{Ric}}
\def\bg{\bar g}
\def\gE{g_{_E}}
\def\bRic{\overline{\Ric}}  
\def\So{\Sigma_{_O}}      
\def\Sh{\Sigma_{_H}}      
\def\fg{\breve{g}}              
\def\fH{{H}}                       
\def\m{\mathfrak{m}}        
\def\N{\mathbb{N}}           
\def\Hb{\bar H}                    
\def\Ab{\bar A}                    
\def\dvol{d \mathrm{vol}}   
\def\mB{\m_{_B}}
\def\E{\mathbb{E}}

\def\tXr{\tilde X_r}
 
\def\Xz{X_{\sigma}}

\begin{document}

\title{Minimal hypersurfaces and Boundary behavior \\ of compact manifolds with   nonnegative \\
 scalar curvature}

\author[Siyuan Lu]{Siyuan Lu}
\address{Department of Mathematics and Statistics, McGill University,  805 Sherbrooke O, Montreal, 
Quebec, Canada, H3A 0B9.}
\email{siyuan.lu@mail.mcgill.ca}

\author[Pengzi Miao]{Pengzi Miao}
\address{Department of Mathematics, University of Miami, Coral Gables, FL 33146, USA.}
\email{pengzim@math.miami.edu}

\thanks{The first named author's research was partially supported by  CSC fellowship. The second named author's research was partially supported by the Simons Foundation Collaboration Grant for Mathematicians \#281105.}

\begin{abstract}
On a  compact Riemannian manifold with boundary having positive mean curvature, 
a fundamental result of  Shi and Tam  states that, 
if the manifold has nonnegative scalar curvature and 
if the boundary is isometric to a strictly convex hypersurface in the Euclidean space, then the total mean 
curvature of the boundary  is no greater than the total mean curvature 
of the corresponding Euclidean hypersurface. In $3$-dimension, 
Shi-Tam's result is known to be equivalent to the Riemannian positive mass theorem.

In this paper, we provide  a supplement to Shi-Tam's result by including 
the  effect of minimal hypersurfaces on the boundary.
More precisely, given a compact manifold  $\Omega$ with nonnegative scalar curvature, 
assuming  its boundary consists  of two parts, $ \Sh $  and $ \So$, where 
$\Sh$ is the union of all closed minimal hypersurfaces in $\Omega$ and $ \So$ is isometric 
to a suitable $2$-convex hypersurface $\Sigma$ in a  spatial Schwarzschild manifold of positive mass $m$, 
we establish an inequality relating $m$, the area of $ \Sh$, and 
two weighted total mean curvatures  of $\So$ and $ \Sigma$.
In $3$-dimension, the inequality has implications to both isometric embedding 
and quasi-local mass problems. 
In a relativistic context, 
our result can be interpreted as a quasi-local mass type quantity of $ \So$ 
being greater than or equal to  the Hawking mass of $\Sh$. 
We further analyze  the limit of such  quasi-local mass  quantity associated with  suitably  chosen  
isometric embeddings of large coordinate spheres of an asymptotically flat $3$-manifold $M$
into a spatial Schwarzschild manifold.  We show that the limit equals the ADM mass of $M$.
It follows that our result on the  compact manifold $\Omega$ is equivalent to 
the Riemannian Penrose inequality.
\end{abstract}

\maketitle 

\markboth{Siyuan Lu and Pengzi Miao}{Boundary behavior of compact manifolds}

\section{Introduction and statement of results}

The main goal of this paper is to prove the following theorem:

\begin{theo}  \label{thm-intro-main}
Let  $(\Omega^{n+1}, \fg)$ be a compact, connected,  orientable, $(n+1)$-dimensional Riemannian manifold 
with nonnegative scalar curvature, with boundary $\p \Omega$. 
Suppose $ \p \Omega$ is the disjoint union of two pieces, $ \So $ and $ \Sh$, where 
\begin{enumerate}
\item[(i)]   $\So$ has positive mean curvature $\fH$; and  
\item[(ii)]  $\Sh$,  if nonempty, is a minimal hypersurface
(with one or more components)
and there are no other closed  minimal hypersurfaces in $ (\Omega, \fg)$.  
\end{enumerate}
Let $ \SM$ denote an  $(n+1)$-dimensional spatial Schwarzschild manifold,  
outside the  horizon, of  mass $ m >0 $ . 
Suppose $ \So$ is isometric to a closed, star-shaped, $2$-convex hypersurface 
$ \Sigma^n \subset \SM$ with $ \bRic(\nu, \nu) \le 0 $, where ${\bRic}$ is the 
Ricci curvature of $\SM$ and $ \nu $ is the outward unit normal to $ \Sigma$. 

If $ n < 7 $, then 
\be \label{eq-intro-main}
m + \frac{1}{n \omega_n} \int_\Sigma NH_m \, d \sigma  \geq 
 \frac12  \left(\frac{|\Sh|}{\omega_n}\right)^{\frac{n-1}{n}}  + \frac{1}{n \omega_n} \int_{\So} N \fH \, d \sigma . 
\ee
Here $H_m  $ is the mean curvature of $ \Sigma$ in $ \SM$, $ d \sigma$ is  the area element on $ \Sigma$ and $ \So$, 
$\omega_n$  is the area  of the standard unit sphere $\mathbb{S}^n$, 
$N$ is the static potential function on $ \SM$  
given by 
$$
N = \frac{ 1 - \frac{m}{2} | x|^{1-n} }{ 1 + \frac{m}{2} |x|^{1 - n } }
$$
if  one writes  
$$\SM = \left( \mathbb{R}^{n+1} \setminus \left\{ | x | <  \left( \frac{m}{2} \right)^\frac{1}{n-1}   \right\}, 
\left( 1 + \frac{m}{2}  | x |^{1-n}  \right)^\frac{4}{n-1} g_{_E} \right)  $$
where $g_{_E} $ is the  Euclidean metric,  
 $ N $ is also viewed as a function on $ \So$ via the isometry  between $\Sigma$ and $ \So $, 
 $ | \Sh |$ denotes the area of $ \Sh $,  and $ | \Sh | $ is taken to be $0$ if $ \Sh = \emptyset $.

Moreover, if equality in \eqref{eq-intro-main} holds, then 
\begin{equation*} \label{eq-intro-equal}
 \fH = H_m  \ \  \mathrm{and} \  \  \frac12  \left(\frac{|\Sh|}{\omega_n}\right)^{\frac{n-1}{n}}  = m . 
\end{equation*} 
In particular, $ \Sh$ must be nonempty  in this case. 
\end{theo}

\begin{rema}
Compact manifolds $(\Omega, \fg)$ satisfying conditions (i) and (ii) in Theorem \ref{thm-intro-main}  exist widely. 
For instance, given any compact, connected,  orientable  Riemannian manifold $(\tilde{\Omega}, \fg)$ 
with disconnected boundary $ \p \tilde \Omega$,  
if the mean curvature vector of $ \p \tilde \Omega$ points  inward at each boundary component,  
then by minimizing area among all hypersurfaces that bounds a domain with a chosen boundary component, 
one  can always construct such an $(\Omega, \fg)$ (under the  given dimension assumption). 
In a relativistic context, a  compact manifold $(\Omega, \fg)$ satisfying conditions (i) and (ii) 
represents a finite body surrounding the apparent horizon of the black hole in a time-symmetric initial data set. 
\end{rema}

\begin{rema}
Let $ \Sh^S = \p \SM$ be the minimal hypersurface boundary of $ \SM$. 
Using the fact $ m = \frac12  \left(\frac{|\Sh^S|}{\omega_n}\right)^{\frac{n-1}{n}} , $
we can write  \eqref{eq-intro-main} equivalently as 
\be \label{eq-intro-main-h}
 \frac12  \left(\frac{|\Sh^S|}{\omega_n}\right)^{\frac{n-1}{n}}  + \frac{1}{n \omega_n} \int_\Sigma NH_m \, d \sigma  \geq 
 \frac12  \left(\frac{|\Sh|}{\omega_n}\right)^{\frac{n-1}{n}}  + \frac{1}{n \omega_n} \int_{\So} N \fH \, d \sigma . 
\ee
Such an  inequality has the following variational interpretation. 
Let $g$ denote the induced metric on $\Sigma $ from the Schwarzschild metric $\bg$ on $\SM$.
Let $ \mathring{\mathcal{F}}_{(\Sigma, g)} $ be the set of {\em fill-ins of $(\Sigma, g)$ 
with outermost horizon inner boundary}, i.e. 
$ \mathring{\mathcal{F}}_{(\Sigma, g)} $ consists of all compact, connected, orientable  
manifolds $(\Omega, \fg)$ with nonnegative scalar curvature, with boundary satisfying 
(i) and (ii) such that  $ \So = \Sigma$ and $ \fg |_{ \So} = g $, where $ \fg |_{\So} $ is  the induced metric 
on $ \So$ from $\fg$. Let $ N$ be the function on $ \So  = \Sigma$, which is the restriction of the static potential
on $\SM$ to $\Sigma$. On $\mathring{ \mathcal{F}}_{(\Sigma, g)}$, 
consider the functional 
$$ (\Omega, \fg)  \longmapsto \frac12  \left(\frac{|\Sh|}{\omega_n}\right)^{\frac{n-1}{n}}  
+ \frac{1}{n \omega_n} \int_{\So} N \fH \, d \sigma .  $$
Inequality \eqref{eq-intro-main-h}  asserts that this functional is maximized at $(\Omega^S, \bg ) $,
where $\Omega^S$ is the domain in $\SM$  bounded by $\Sigma$ and $ \Sh^S$. 
(Such an interpretation of \eqref{eq-intro-main-h} in terms of fill-ins relates  to 
the work of Mantoulidis and 
the second author 
\cite{Mantoulidis-Miao}.) 
\end{rema}

Treating the assumption that  $ \So $ is isometric to  $ \Sigma \subset  \SM$
as a condition of having an  isometric embedding of $ \So $ into $ \SM$, 
 we  have  the following result. 

\begin{theo} \label{thm-intro-iso-emb}
Let $ (M^3, \fg)$ be  an asymptotically flat $3$-manifold.
Let $ S_r $ denote the  coordinate sphere of coordinate radius $r$
in a  coordinate chart  defining  the asymptotic flatness of $(M, \fg)$. 
Let $ g_r$ be the induced metric on $ S_r$.
Then, given any  constant $ m > 0 $, 
there exists an isometric embedding 
$$X_r: (S_r, g_r) \longrightarrow \mathbb{M}^3_m  $$ for each sufficiently large $r$, such that   
$\Sigma_r = X_r ( S_r)$ is a star-shaped, convex surface in $\mathbb{M}^3_m$,  with  $ \bRic (\nu, \nu) < 0 $
where $ \nu$ is the outward unit normal to $ \Sigma_r$;
moreover,  
\be \label{eq-intro-nby-limit}
  \lim_{r \to \infty}  \left( m + \frac{1}{8 \pi} \int_{S_r} N ( H_m - H) \, d \sigma  \right) =  \m , 
\ee
and
\be    \label{eq-intro-volume-limit}
 V(r) - V_m (r)    =    2 \pi r^2 (\m -m) + o (r^2 ) , \ \mathrm{as} \ r \to \infty.  
\ee 
Here $ \m $ is the ADM mass of $(M, \fg)$,
 $H$ is  the mean curvature of $S_r$ in $(M, \fg)$ and  $H_m$ is   the mean curvature of $ \Sigma_r$ in 
$\mathbb{M}^3_m$, 
$N$ is the static potential on $\mathbb{M}^3_m$,
 $N$ and $H_m$ are viewed as functions on $ S_r$ via the embedding $X_r$,
 $V(r)$  is  the volume of the region enclosed by $ S_r$ in $(M, \fg)$ and 
  $V_m (r)$ is the volume of the region enclosed by $ \Sigma_r$ in $ \mathbb{M}^3_m$.
\end{theo}

Now we explain the motivations to and the implications of Theorem \ref{thm-intro-main}. 
Our first motivation  to Theorem \ref{thm-intro-main} is  the following  
theorem of  Shi and Tam \cite{ST}.

\begin{theo} [\cite{ST}] \label{thm-intro-ST}
Let $ (\tilde{\Omega}^{n+1}, \fg)$ be a compact,  Riemannian spin manifold 
with nonnegative scalar curvature, with boundary $\p \tilde \Omega$. 
Let $ \Sigma_i$, $1 \le i \le k $, be the  connected components of $ \p \tilde \Omega$. 
Suppose each $ \Sigma_i $  has positive mean curvature 
and each $ \Sigma_i$ is  isomeric to a strictly convex hypersurface 
$ \hat \Sigma_i  \subset \mathbb{R}^{n+1}$.
Then
\be \label{eq-intro-ST}
\int_{\hat \Sigma_i} H_0 \, d \sigma \ge \int_{\Sigma_i}  \fH \, d \sigma , 
\ee
where $ H_0 $ is the mean curvature of $ \hat \Sigma_i $ in $ \mathbb{R}^{n+1}$
and $ H$ is the mean curvature of $ \Sigma_i$ in $ ( \tilde \Omega, \fg)$. 
Moreover, if equality holds for some $i$, then $k =1 $ and 
$( \tilde \Omega, \fg )$ is isometric to a domain in $ \mathbb{R}^{n+1}$.  
\end{theo} 

Theorem \ref{thm-intro-ST} is a  fundamental  result on compact manifolds 
with nonnegative scalar curvature with boundary,
obtained via  the Riemannian positive mass theorem \cite{SY79, Witten81}.
 For the purpose of  later explaining  the proof  of Theorem \ref{thm-intro-main},  
 we outline the  proof of Theorem \ref{thm-intro-ST} from  \cite{ST} as follows. 
 For simplicity, we assume  $k=1$ and 
 denote $ \Sigma_1 $ by  $\Sigma$. 
 Identifying $\Sigma$ with its isometric image in $ \mathbb{R}^{n+1}$  
and using the assumption that  $ \Sigma $ is  convex in $ \R^{n+1}$, 
one can write the Euclidean metric $ \gE$ on $\E$,  the exterior of $ \Sigma $, 
as  $ \gE = d t^2 + g_t $, where $g_t $ is  the induced metric  on the  hypersurface $\Sigma_t$  that 
has a fixed Euclidean distance $t$  to $ \Sigma$.
Given the mean curvature function $ \fH > 0 $ on $ \Sigma$, one 
shows that there exists a function $ u > 0 $ on $ \E$ such that
$ g_u = u^2 d t^2 + g_t $ has zero scalar curvature, $(\E, g_u) $ is asymptotically flat, and the mean curvature 
$H_u $ of $ \Sigma_t $ in $ (\E, g_u)$ 
satisfies $ H_u =  \fH$ at $ \Sigma_0 = \Sigma$. 
A key feature of such  an $(\E, g_u)$ is that the integral 
\be \label{eq-intro-BY}
\frac{1}{ n \omega_n}  \int_{\Sigma_t} (H_0 - H_u ) \, d \sigma 
 \ee
is monotone nonincreasing  and it  converges to $ \m (g_u) $, 
where $\m (g_u) $ is the ADM mass \cite{ADM} of $ (\E, g_u)$. 
By gluing $(\tilde \Omega, \fg) $ and $(\E, g_u)$ along their common boundary $\Sigma$ 
and applying  the Riemannian positive mass theorem,
which is  still valid under the condition that the mean curvatures of $\Sigma$ in $(\tilde \Omega, \fg)$ and $ (\E, g_u)$ agree 
(see \cite{ST, Miao02}), one concludes that
\be \label{eq-intro-pf-ST}
\frac{1}{n \omega_n}  \int_\Sigma (H_0 - H) \, d \sigma \ge \lim_{t \rightarrow \infty} \frac{1}{n \omega_n}   
\int_{\Sigma_t}  (H_0 - H) \, d \sigma  =  \m (g_u) \ge 0  ,
 \ee
which proves \eqref{eq-intro-ST}.

One of the most important features of Theorem \ref{thm-intro-ST} is that,
when  $n=2$, by  the  solution to the Weyl embedding problem (\cite{Nirenberg, Pogorelov}),  
Theorem \ref{thm-intro-ST} implies the positivity of  the Brown-York quasi-local mass (\cite{BY1, BY2})
of $ \p \tilde \Omega$, under the  assumption that  
$ \p  \tilde \Omega$ is a topological  $2$-sphere with positive  Gauss curvature.  

\begin{rema}
When $ n > 2 $, Eichmair,  Wang and the second author \cite{EMW} proved that 
Theorem \ref{thm-intro-ST}   remains  valid  if 
each component $ \Sigma_i$ is isometric to a  star-shaped hypersurface  with positive scalar curvature in $ \mathbb{R}^{n+1}$.
It was also noted in \cite{EMW} that  the spin assumption therein  can be dropped  when $  n < 7 $.
Recently, Schoen and Yau \cite{SY17} proved  that the Riemannian positive mass theorem holds 
in all dimensions without a spin assumption. Therefore, by the argument in \cite{EMW}, 
results in \cite{ST, EMW}  also  hold in all dimensions without a spin assumption.
\end{rema}

To motivate Theorem \ref{thm-intro-main} from Theorem \ref{thm-intro-ST}, 
one may  consider the setting  $ k > 1 $ of 
Theorem \ref{thm-intro-ST}. In this case,  given  any boundary component
$\Sigma_i$,  there  exists a minimal hypersurface $S_i$, possibly disconnected, 
in the interior of $(\tilde \Omega, \fg)$ such that $ S_i$ and $ \Sigma_i$ bounds a domain $\Omega$ 
satisfying conditions (i) and (ii) in Theorem \ref{thm-intro-main}. 
Thus, besides the nonnegative scalar curvature, one wants to understand the influence of $ S_i$ on $ \Sigma_i$.
This is indeed related to the following Riemannian Penrose inequality, 
which is our second motivation to Theorem \ref{thm-intro-main}. 

\begin{theo}[\cite{HI01, B, BL}] \label{thm-intro-RPI}
Let $M^{n+1}$ be an asymptotically flat manifold with nonnegative scalar curvature, 
with boundary $\p M$, where $  n < 7$. 
Suppose $ \p M $ is an outer minimizing, minimal hypersurface (with one or more component), then
\be  \label{eq-intro-RPI}
\m (M)  \geq \frac12\left( \frac{|\p M |}{\omega_{n}}  \right)^{\frac{n-1}{n}}, 
\ee 
where $ \m (M) $ is the ADM mass of $M$ and $ | \partial M | $ is  the area of $ \partial M $.  
Moreover, equality holds if and only if $M$ is isometric to a spatial Schwarzschild manifold outside its horizon. 
\end{theo}

When $ n =2$,  Theorem \ref{thm-intro-RPI} was first proved by Huisken and Ilmanen \cite{HI97, HI01}  
for the case that $ \p M$  is connected, and later proved  by  Bray \cite{B} for  the general case  
in which  $ \p M $ can have multiple components.
For higher dimensions,  Bray and Lee  \cite{BL} proved  inequality \eqref{eq-intro-RPI}  for $ n < 7 $ 
and established the rigidity case assuming  that $M$ is spin.
(Without the spin assumption, the rigidity case follows by combining results of  Bray and Lee \cite{BL} 
and  McFeron and Sz\'{e}kelyhidi \cite{MS}.)

To compare Theorem \ref{thm-intro-main} and Theorem \ref{thm-intro-RPI}, 
we can write \eqref{eq-intro-main} equivalently as
\be \label{eq-intro-main-sh}
m + \frac{1}{n \omega_n} \int_{\So} N ( H_m - \fH)  \, d \sigma  \geq 
 \frac12  \left(\frac{|\Sh|}{\omega_n}\right)^{\frac{n-1}{n}} 
\ee
by identifying $ \So$ and $ \Sigma$.
The quantity  on the  left side of \eqref{eq-intro-main-sh} depends only on 
the assumption on the (outer) boundary component $\So$ of  $ \Omega$, 
while the mass $\m (M)$ in  \eqref{eq-intro-RPI} is determined solely   
by the asymptotically flat end of $M$. 
In this sense, Theorem \ref{thm-intro-main} can be viewed as a localization of 
Theorem \ref{thm-intro-RPI}  to a compact manifold with boundary satisfying conditions (i) and (ii). 
 Indeed, by \eqref{eq-intro-nby-limit} in Theorem \ref{thm-intro-iso-emb}
 and the fact that our proof of Theorem \ref{thm-intro-main}  uses  \eqref{eq-intro-RPI}, 
 Theorem \ref{thm-intro-main}  is  equivalent to the Riemannian Penrose inequality \eqref{eq-intro-RPI} when $ n = 2$.
In this case, the right side of \eqref{eq-intro-main-sh} is  the Hawking quasi-local mass \cite{Hawking} of $\Sh$, 
and \eqref{eq-intro-main-sh} describes how $\Sh$, which models the apparent horizon of  black hole, 
contributes to the quasi-local mass of  a  body surrounding it.

\begin{rema}
In \cite{CWWY},  Chen, Wang, Wang and Yau introduced a notion 
of quasi-local energy in reference to a general static spacetime. 
Setting $ \tau = 0 $ in equation (2.10) in \cite{CWWY}, 
one sees that the quasi-local energy of a $2$-surface $ \Sigma $ defined in  \cite{CWWY} 
with respect to an isometric embedding of $\Sigma$ into  a time-symmetric slice of Schwarzschild
the Schwarzschild spacetime with mass $m$ is given by 
$  \frac{1}{8 \pi } \int_{\Sigma} N ( H_m - \fH)  \, d \sigma $,
which agrees with the surface integral  on the left side of \eqref{eq-intro-main-sh} 
with $ \Sigma = \So$. 
\end{rema}

To illustrate  that Theorem \ref{thm-intro-main} provides a  supplement to 
Shi-Tam's result,  we want to make a connection  between   \eqref{eq-intro-main-sh} and 
an inequality that can be obtained by directly combining  \eqref{eq-intro-RPI} 
and Shi-Tam's proof of  Theorem \ref{thm-intro-ST}. 
Only for the convenience of making  a comparison, we list
the following inequality in a theorem format:

\vspace{.2cm}

\noindent {\bf Theorem 1.3'}
{\em Let $ (\Omega^{n+1}, \fg) $ be a compact Riemannian manifold with nonnegative scalar curvature, with boundary $ \p \Omega$, 
satisfying conditions (i) and (ii) in Theorem \ref{thm-intro-main}. 
Suppose $ \Sh \neq \emptyset $  and  $\So$ is isometric to a strictly convex hypersurface $ \Sigma^n \subset \R^{n+1}$.
 If $ n < 7$, then 
\be \label{eq-intro-ST-sh}
 \frac{1}{n \omega_n} \int_{\So} ( H_0 - \fH)  \, d \sigma  > 
 \frac12  \left(\frac{|\Sh|}{\omega_n}\right)^{\frac{n-1}{n}} ,
\ee
where $H_0$ is the mean curvature of $\Sigma$ in $ \R^{n+1}$. 
}

\vspace{.2cm} 

The proof of  \eqref{eq-intro-ST-sh}  is identical to Shi-Tam's proof of Theorem \ref{thm-intro-ST} outlined earlier, 
except that in the final  inequality of \eqref{eq-intro-pf-ST}, one replaces 
the Riemannian positive mass theorem 
by the Riemannian Penrose inequality to yield 
\be \label{eq-intro-pf-ST-sh}
\frac{1}{n \omega_n}  \int_{\So} (H_0 - H) \, d \sigma \ge \lim_{t \rightarrow \infty}  \frac{1}{n \omega_n}  \int_{\Sigma_t}  (H_0 - H) \, d \sigma  
=  \m (g_u) \ge    \frac12  \left(\frac{|\Sh|}{\omega_n}\right)^{\frac{n-1}{n}} .
 \ee
The fact  that  \eqref{eq-intro-RPI} is applicable to the manifold  obtained by gluing $(\Omega, \fg)$ and $(\E, g_u)$
was demonstrated  in \cite{Miao09} for $n=2$ and in \cite{MM} for $  n < 7$. 

\begin{rema}
By the argument in \cite{EMW},  
 \eqref{eq-intro-ST-sh} holds with  the assumption that $\Sigma \subset \R^{n+1}$ is strictly convex replaced by 
 that $ \Sigma$ is star-shaped with positive scalar curvature. 
 Such a statement  is precisely   the $m=0$ analogue of Theorem \ref{thm-intro-main} for the case 
$\Sh \neq \emptyset$.
\end{rema}

Inequality  \eqref{eq-intro-ST-sh} takes a simpler form than \eqref{eq-intro-main-sh},
however  it is always a strict inequality. 
This is because, if the first inequality in \eqref{eq-intro-pf-ST-sh} were equality, 
the function $u$ would be identically $1$ (implied by the monotonicity calculation of 
\eqref{eq-intro-BY} in \cite{ST, EMW}),  consequently  $ H_0 = \fH$ identically,  
 which would show  $ 0 \ge | \Sh | $, contradicting the assumption  $ \Sh \neq \emptyset$. 
A more intuitive reason for \eqref{eq-intro-ST-sh} to be  strict  is that, though $ \Sh$ is a nonempty 
minimal hypersurface in $ \Omega^{n+1}$,   \eqref{eq-intro-ST-sh} is obtained by comparing $ \So $ 
to a hypersurface in $ \R^{n+1}$ which is free of  closed minimal hypersurfaces. 

For the above reason, we consider an assumption  $ \So$ is isometric to an $ \Sigma \subset \SM$
in Theorem \ref{thm-intro-main}. 
In particular, \eqref{eq-intro-main-sh} does become an equality when $\Omega$ itself  is  the domain 
in $ \SM$ bounded by  $\Sigma$ and the Schwarzschild horizon $\Sh^S$.

The fact that \eqref{eq-intro-main-sh} gives a  refined estimate  on $ | \Sh |$, sharper than \eqref{eq-intro-ST-sh}, 
can be  illustrated by the  case  in which   $\So$ is  isometric to a round sphere. 
In the following example, for simplicity, we  take $n=2$.

\begin{exam}
Suppose  $ \Omega $ is a compact $3$-manifold with nonnegative scalar curvature,
with boundary $\p \Omega$, satisfying conditions (i) and (ii) in Theorem \ref{thm-intro-main}.
Suppose  $ \Sh \neq \emptyset $ and $ \So $  is isometric to a round sphere with  area $ 4\pi R^2$.
Then  \eqref{eq-intro-ST-sh}  shows 
 \be \label{eq-exam-ST-sh}
  R -  \frac{1}{8\pi} \int_{\So}  \fH \, d \sigma >  \sqrt{ \frac{ | \Sh |}{16 \pi} } .
 \ee 
On the other hand, Theorem \ref{eq-intro-main}  applies to any  $ \mathbb{M}_m^3$ with
 $ m \in \left(0, \frac12 R \right)$ 
 since $ \So $ is  isometric to a rotationally symmetric sphere   in such an $\mathbb{M}^3_m $. 
Thus,  by  \eqref{eq-intro-main-sh}, 
\be \label{eq-intro-main-sh-so-round}
m + \frac{1}{8 \pi}  \int_{\So} N  \left( N \frac{2}{R}  - \fH \right) \, d \sigma \ge  \sqrt{ \frac{ | \Sh |}{16 \pi} }  
\ee
with $N = \sqrt{ 1 - \frac{2m}{R} }$. 
Let $ \Phi (m) $ denote the quantity on the left side of \eqref{eq-intro-main-sh-so-round}. 
(The left side of \eqref{eq-exam-ST-sh}  equals $ \lim_{m \to 0+} \Phi (m)$.) 
By  \eqref{eq-intro-main-sh-so-round}, 
\be 
 \min_{ 0 < m < \frac{R}{2} } \Phi (m) \ge  \sqrt{ \frac{ | \Sh |}{16 \pi} } . 
\ee
Note that either  \eqref{eq-exam-ST-sh} or \eqref{eq-intro-main-sh-so-round} implies 
 $ 0 <  \frac{1}{8 \pi R} \int_{\So} \fH \, d \sigma < 1 $.
Therefore, via direct calculation, one has 
\be \label{eq-intro-round-case}
\begin{split}
 R -  \frac{1}{8\pi} \int_{\So}  \fH \, d \sigma  > & \  \min_{ 0 < m < \frac{R}{2} } \Phi (m)  
=   \frac{R}{2} \left[ 1 - \left(  \frac{1}{8 \pi R} \int_{\So} \fH \, d \sigma  \right)^2 \right] \\
\ge  & \  \sqrt{ \frac{ | \Sh |}{16 \pi} }  .
\end{split} 
\ee
(It is clear that, if $ \Omega$ is the region bounded by a rotationally symmetric sphere 
and the horizon boundary in some $\mathbb{M}^3_m $, then 
$ \min_{ 0 < m < \frac{R}{2} } \Phi (m)   =  \sqrt{ \frac{ | \Sh |}{16 \pi} }  $.)
In \eqref{eq-intro-round-case},  it is also intriguing  to 
note that  $ \min_{ 0 < m < \frac{R}{2} } \Phi (m) $ is achieved
at $ m = m_*$  where $m_*$, 
determined by $ N = \frac{1}{8 \pi R} \int_{\So} \fH \, d \sigma  $,
agrees with $ \min_{ 0 < m < \frac{R}{2} } \Phi (m) $, i.e. 
\be \label{eq-intro-optimal-m}
 m_* = \min_{ 0 < m < \frac{R}{2} } \Phi (m)   .
 \ee
This means that an optimal background $\mathbb{M}^{3}_{m_*} $ that is used to be compared with 
 $\Omega$ is indeed  determined by the minimal value of $\Phi (m)$. 
 \end{exam}

\begin{rema}
Calculation in relation to the example above 
was first carried out  in \cite{Miao09} where the special  case of  Theorem \ref{thm-intro-main} 
in which  $\So $ is  isometric to  a round sphere was proved.
The implication of \eqref{eq-intro-optimal-m} on the quasi-local mass of such round surfaces
 was also discussed in \cite{Miao09}. 
\end{rema}

Next, we comment on the implication of Theorem \ref{thm-intro-main}  on isometric embeddings 
of a $2$-sphere into a  Schwarzschild manifold  $ \mathbb{M}^{3}_{m} $ with $ m > 0$.
It was proved by Li and Wang \cite{LW} that, if $\sigma$ is a metric on the $2$-sphere $ \mathbb{S}^2$, 
an isometric embedding of  $(\mathbb{S}^2, \sigma) $ into $ \mathbb{M}^3$ may not be unique. 
Indeed, it was shown  in \cite{LW} that, if $ \sigma_r$ is the  standard round metric of area $ 4 \pi r^2 $ 
with  $ r > 2 m $, then $ (\mathbb{S}^2, \sigma_r)$ admits 
 an isometric embedding into $\mathbb{M}^3_m$ that is  close to but different from the 
 standard embedding whose image is a rotationally symmetric sphere. 
For this reason, one knows that  inequality \eqref{eq-intro-main} 
does depend on the choice of the isometry between $\So $ and $ \Sigma$. 
 (This contrasts with inequality \eqref{eq-intro-ST}  which only depends on the 
 intrinsic metric on $\Sigma_i$.)
However, in the following example, we demonstrate that \eqref{eq-intro-main}  can be  applied to 
reveal  information on such different isometric embeddings into $\mathbb{M}^{3}_{m}$. 

\begin{exam}
Let $ \Sigma \subset \mathbb{M}^3_m$ be a closed, star-shaped, convex surface 
 with $ \bRic(\nu, \nu) \le 0 $. 
Let $ H_m$ denote its mean curvature. 
Suppose $\iota: \Sigma \rightarrow \tilde \Sigma $ is an isometry between $\Sigma$ 
and another surface 
$ \tilde \Sigma  \subset \mathbb{M}^3_m $ with properties 
\begin{itemize}
\item[(a)] $ \tilde \Sigma $ bounds a domain $D$ with the Schwarzschild horizon  
$ \Sh^S = \p \mathbb{M}^3_m$, and
\item[(b)] $ \tilde \Sigma$ has positive mean curvature $\tilde H_m$ with respect to the outward unit normal.
\end{itemize}
Then Theorem \ref{thm-intro-main} is applicable to the domain $D$ to give
\be \label{eq-intro-main-exam}
m + \frac{1}{8 \pi } \int_\Sigma NH_m \, d \sigma  \geq 
\sqrt{  \frac{ | \Sh^S | }{ 16 \pi } }  + \frac{1}{8 \pi } \int_{\tilde \Sigma} \tilde{N} \tilde H_m  \, d \sigma 
\ee
with $ \tilde N = N \circ \iota^{-1}  $. 
(Note that, if (a) is replaced by an assumption $ \tilde \Sigma = \partial D$ for some $D$, then the term
involving  $| \Sh^S|$  will be absent in \eqref{eq-intro-main-exam} and the inequality is strict.)
Since $ m = \sqrt{  \frac{ | \Sh^S | }{ 16 \pi } } $, 
\eqref{eq-intro-main-exam} shows 
\be \label{eq-intro-main-exam-1}
 \int_\Sigma NH_m \, d \sigma  \geq \int_{\tilde \Sigma} \tilde{N} \tilde H_m  \, d \sigma ,
\ee
with equality holds only if $H_m \circ \iota^{-1} = \tilde H_m $. 
Now suppose we consider the special case in which $ \Sigma$ is a rotationally symmetric sphere, 
then $ N $ is a constant on $ \Sigma$, hence $ \tilde N $ is also a constant that equals $N$. 
In this case, \eqref{eq-intro-main-exam-1} becomes 
\be \label{eq-intro-main-exam-2}
 \int_\Sigma H_m \, d \sigma  \geq \int_{\tilde \Sigma} \tilde H_m  \, d \sigma .
\ee
(In the case of $\tilde \Sigma = \partial D$, one has 
$ 8\pi m N^{-1} +  \int_\Sigma H_m \, d \sigma  >  \int_{\tilde \Sigma} \tilde H_m  \, d \sigma$.)
Since  $ H_m$ is a constant, equality in \eqref{eq-intro-main-exam-2} 
holds only if $\tilde H_m $ is a constant. 
By the result of Brendle \cite{Brendle-CMC}, we conclude that  $\tilde \Sigma$ must be  $\Sigma$
when  equality holds in  \eqref{eq-intro-main-exam-2}.
\end{exam}

We now outline the proof of Theorem \ref{thm-intro-main}. 
The first step in our proof is to generalize the monotonicity of the Brown-York mass type integral \eqref{eq-intro-BY}
in Shi-Tam's proof of Theorem \ref{thm-intro-ST} to the monotonicity of a weighted Brown-York mass type   integral 
\be \label{eq-intro-w-BY}
 \int_{\Sigma_t} N ( \Hb - H_\eta ) \, d \sigma 
\ee
in a general static background on which $N$ is a positive static potential function. 
The idea of imposing a suitable weight function in  \eqref{eq-intro-w-BY} to obtain monotonicity 
goes back to the work of Wang and Yau \cite{WY07} in which isometric embeddings of surfaces 
into hyperbolic spaces are considered. 
Given a static Riemannian manifold $(\N, \bar g)$ (see Definition \ref{def static}), 
let $\{ \Sigma_t \}$ be a family of closed hypersurfaces  evolving in $(\N, \bar g)$ with speed $f > 0 $,
we show that, as long as $ \Sigma_t $ is $2$-convex and $ \frac{\p N}{\p \nu} > 0 $, 
\eqref{eq-intro-w-BY} is monotone nonincreasing along the flow. 
Here  $2$-convexity of $ \Sigma_t $ means that $\sigma_1> 0 $ and $ \sigma_2 > 0 $, 
where $\sigma_1$ and $ \sigma_2$  are the first and second elementary symmetric
functions of the principal curvatures of $\Sigma_t$ in $(\N, \bar g)$; $ \nu $ denotes 
the unit normal giving the direction of the flow; and $\Hb$, $H_\eta$ denote the mean curvature of 
$\Sigma_t$ with respect to $\bar g = f^2 d t^2 + g_t $, $ g_\eta = \eta^2 d t^2 + g_t $,  respectively,
where $ g_\eta $ is taken to have the same scalar curvature as $\bar g$.
(The idea of considering such a metric $g_\eta$ goes back to Bartnik \cite{Bartnik-93}.)
To apply this  monotonicity formula, in the next step we study  
a family of closed, star-shaped,  hypersurfaces $\{ \Sigma_t  \}$  in a spatial Schwarzschild manifold $\SM$,
given by  $ \Sigma_t = X (t, \mathbb{S}^n)$, where $ X: [0, \infty)  \times \mathbb{S}^n \rightarrow \SM $ 
is a smooth map  evolving  according to 
\be \label{eq-intro-flow-sigma}
\frac{\p X}{\p t} = \frac{n-1}{2n} \frac{\sigma_1}{\sigma_2} \nu .
\ee
We shows that, if the initial hypersurface $ \Sigma_0$ is $2$-convex with $ \bRic (\nu, \nu) \le 0 $, then 
\eqref{eq-intro-flow-sigma} admits a long time solution $\{ \Sigma_t \}_{ 0 \le t < \infty} $ so that  each 
$ \Sigma_t $ is $2$-convex and has positive scalar curvature. 
Writing the Schwarzschild background metric $\bar g$ on the exterior   region $\E$ of $ \Sigma_0$ as 
$ \bar g = f^2 dt^2 + g_t $, we then demonstrate that there exists a positive function $ \eta  $ on $ \E $ such that
$ g_\eta = \eta^2 d t^2 + g_t $ has zero scalar curvature, the mean curvature of $\Sigma_0$ in $(\E, g_\eta)$ equals $H$ 
which is  the mean curvature of $ \So $ in $(\Omega, \fg)$; and $(\E, g_\eta)$ is asymptotically flat with mass 
\be \label{eq-intro-gu-mass}
 \m (g_\eta) = m +  \lim_{t \rightarrow \infty} \frac{1}{n \omega_n} \int_{\Sigma_t} N ( \Hb - H_\eta ) \, d \sigma  . 
 \ee
 Finally, by gluing $(\Omega, \fg)$ and $ (\E, g_\eta)$ along $\So $ (which is identified with $\Sigma =  \Sigma_0$)
 to get an asymptotically flat manifold $(\hat M, \hat h)$, we conclude 
 \begin{equation} \label{eq-intro-app-RPI}
 \begin{split}
m + \frac{1}{n \omega_n} \int_{\So} N ( H_m - H) \, d \sigma  \ge  & \ 
m + \lim_{t \rightarrow \infty} \frac{1}{n \omega_n} \int_{\Sigma_t} N ( \Hb  - H_\eta) \, d \sigma \\
= & \ \m (g_\eta)  \ge   \frac12  \left(\frac{|\Sh|}{\omega_n}\right)^{\frac{n-1}{n}} ,
 \end{split} 
 \end{equation}
 where in the last step we 
 used the fact that the Riemannian Penrose inequality holds on such an $(\hat M, \hat h)$
 (see \cite{Miao09, MM}).
  
It is worth of mentioning that, similar to the fact that Shi-Tam's proof of Theorem \ref{thm-intro-ST}  
gives   an upper bound of the  Bartnik mass $\mB(\Sigma)$  \cite{Bartnik} for a $2$-surface $\Sigma$
that  is isometric to a convex surface in $ \R^3$  in terms of  its Brown-York mass, 
our proof of Theorem \ref{thm-intro-main} yields  
\be \label{eq-intro-B-mass}
\mB(\Sigma) \le   m +  \frac{1}{8 \pi} \int_{\Sigma} N ( H_m - H) \, d \sigma  
\ee
for a surface $\Sigma$  that is isometric to a convex  surface  with $ \bRic(\nu, \nu) \le 0 $ 
in an $ \mathbb{M}^3_m$ (see Theorem \ref{thm Bmass}).
Such an estimate on the Bartnik mass 
verifies  a special case of Conjecture 4.1 in \cite{Miao-ICCM-07}, which is formulated  
for a surface that admits an isometric embedding into a general static manifold.

This paper is organized as follows. In Section \ref{sec-mono-static}, we derive the monotonicity formula of the
weighted Brown-York mass type integral \eqref{eq-intro-w-BY} in a general static background. 
In Section \ref{sec-flow}, we study a family of inverse curvature flows in a spatial Schwarzschild manifold $\SM$, 
which includes \eqref{eq-intro-flow-sigma} as a special case. In Section \ref{sec-warped-metric}, 
we prove that a warped metric of the form $g_\eta = \eta^2 d t^2 + g_t$, with zero scalar curvature,  
exists  on the Schwarzschild exterior region $ \E$  swept out by the  solution  
$\{ \Sigma_t \}_{0 \le t \le \infty} $ to  \eqref{eq-intro-flow-sigma}, and show that $g_\eta$ is asymptotically flat and its  
mass  is given by \eqref{eq-intro-gu-mass}.  In Section \ref{sec apps}, we attach $(\E, g_\eta)$ to  $ (\Omega, \fg) $ along $ \So$
and apply the Riemannian Penrose inequality to prove Theorem \ref{thm-intro-main}. 
We also discuss  the implication of our work to the Bartnik mass.
We end the paper by  proving Theorem \ref{thm-intro-iso-emb} in Section \ref{sec-iso-emb}. 

\section{Monotonicity formula in a static background}  \label{sec-mono-static}

The Euclidean space $ \R^{n+1}$ and the spatial Schwarzschild manifolds $\SM$ both are examples 
of a static Riemannian manifold according to the following definition.

\begin{defi} [\cite{Corvino}] \label{def static}
A Riemannian  manifold $(\N,\bar{g})$ is called static if there exists a nontrivial function $N$ such that
\begin{align}\label{static equation}
(\bar{\Delta}N)\bar{g}-\bar{D}^2N+N\bar{Ric}=0,
\end{align}
where $\bar{Ric}$ is the Ricci curvature of $(\N, \bar g)$, $\bar{D}^2 N $ is the Hessian of $N$ 
and $\bar{\Delta}$ is the Laplacian of $N$. 
The function $N$ is called a static potential.  
\end{defi} 

Throughout this section, we let  $(\N, \bar g)$ denote a static Riemannian manifold with a  static potential $N$. 
The scalar curvature $\bar R$  of such an $(\N, \bg )$ is necessarily a constant  (see \cite[Proposition 2.3]{Corvino}). 
Consider  a  smooth family of embedded hypersurfaces $\{\Sigma_t \}$  evolving in $(\N, \bg)$ according to 
\begin{align}\label{formula for flow}
\frac{\partial X}{\partial t}=f\nu , 
\end{align}
where $X$ denotes points in $\Sigma_t$,  $f >0$ denotes  the speed  of the flow, and $\nu$ is a  unit normal to $\Sigma_t$.
Let  $\sigma_1$ and $ \sigma_2$ be the first and second elementary symmetric functions of the principal curvatures of $\Sigma_t$ in $(\N, \bar g)$,
respectively. In particular, $ \sigma_1 $ equals  the mean curvature of $\Sigma_t$.

The metric  $\bg $ over the region  $U$ swept by $\{ \Sigma_t \}$ can be written as
\begin{align}\label{metric by foliation}
\bar{g}=f^2dt^2+g_t ,
\end{align}
where $g_t$ is the induced metric of $\Sigma_t$.
Now  consider  another metric 
\begin{align}\label{perturbed metric}
g_\eta=\eta^2dt^2+g_t , 
\end{align}
where $\eta > 0 $ is a function on $U$. 
We impose the condition that the scalar curvature $R(g_\eta)$  of $g_\eta $ 
equals the scalar curvature  of $ \bar g$, i.e. 
\be
R(g_\eta) = \bar R . 
\ee

\begin{prop}  \label{prop-mono}
Under the above notations and assumptions, 
\begin{align*}
  \frac{d}{d t}\left( \int_{\Sigma_t} N ( \Hb - H_\eta ) d \sigma \right)= -  \int_{\Sigma_t} 
  \eta^{-1} ( \eta - f )^2  \Hb  \frac{\partial N}{\partial \nu} d \sigma  - \int_{\Sigma_t}   
  N \sigma_2 \eta^{-1}  (  \eta- f  )^2    d \sigma , 
\end{align*}
where $\Hb$ and $H_\eta$ are the mean curvature of $\Sigma_t$ with respect to $\bar{g}$ 
and $g_\eta$, respectively.
\end{prop}

\begin{proof}
Denote $\Ab$ and $A_\eta$ the second fundamental form of $\Sigma_t$ with respect to  
$\bar{g}$ and $g_\eta$,  respectively. By  \eqref{metric by foliation} and \eqref{perturbed metric}, 
\begin{align}\label{formula H,H_eta}
H_\eta=\eta^{-1}f \Hb, \quad A_\eta=\eta^{-1} f \Ab .
\end{align}
By the second variation formula, 
\begin{align}\label{H}
\frac{\partial}{\partial t} \Hb =  - \Delta f - f( | \Ab |^2 + \bar{Ric}(\nu, \nu)) 
\end{align}
and
\begin{align}\label{H_eta}
\frac{\partial}{\partial t}  H_\eta =  - \Delta \eta - \eta ( | A_\eta |^2 + Ric_{g_\eta} (\nu, \nu)) ,
\end{align}
where $\Delta$ is the Laplacian operator on $(\Sigma_t, g_t)$ and $Ric_{g_\eta}$ is the Ricci curvature of  $g_\eta$.

Let $ R$ denote  the scalar curvature of $(\Sigma_t, g_t)$. 
Let $ {\sigma_2}_\eta $ be the second elementary symmetric functions of the principal curvatures of $\Sigma_t$ in $(\N, g_\eta)$. 
By Gauss equation,
\begin{align}\label{Gauss-1}
\sigma_2=\frac{R-\bar{R}}{2}+\bar{Ric}(\nu,\nu),\quad {\sigma_2}_\eta=\frac{R-\bar{R}}{2}+Ric_{g_\eta}(\nu,\nu) .
\end{align}
Together with (\ref{formula H,H_eta}), we have
\begin{equation}\label{Ric-1}
\begin{split}
Ric_{g_\eta}(\nu,\nu) = & \bar{Ric}(\nu,\nu)+{\sigma_2}_\eta-\sigma_2 \\
= & \bar{Ric}(\nu,\nu)+\sigma_2(\eta^{-2}f^2-1) . 
\end{split}
\end{equation}
Putting  (\ref{H}), (\ref{H_eta}) and (\ref{Ric-1}) together, we have
\begin{align*}
\frac{\partial}{\partial t} ( \Hb  - H_\eta ) = & \Delta ( \eta - f ) - f( | \Ab |^2 + \bar{Ric} (\nu, \nu)) + \eta ( | A_\eta |^2 + Ric_{g_\eta} (\nu, \nu)) \\
=& \Delta ( \eta - f )+\bar{Ric} (\nu, \nu) (\eta-f) + | \Ab |^2(\eta^{-1} f^2-f)++\sigma_2(\eta^{-1}f^2-\eta) .
\end{align*}
Using the formula $\frac{\partial}{\partial t}d \sigma=fHd \sigma$, (\ref{formula H,H_eta}) and integrating  by part, we thus have
\begin{align*}
\frac{d}{dt}\left(\int_{\Sigma_t} N( \Hb  - H_\eta ) d \sigma\right) =&\int_{\Sigma_t}f\frac{\partial N}{\partial \nu} \Hb 
(1-\eta^{-1}f)d \sigma+\int_{\Sigma_t} N \Hb (1-\eta^{-1}f) f \Hb d \sigma\\
&+\int_{\Sigma_t}\left(\Delta N ( \eta - f ) +N \bar{Ric}(\nu, \nu)(\eta-f)\right) d \sigma \\
&+\int_{\Sigma_t}\left( N |\Ab|^2(\eta^{-1}f^2-f)+N\sigma_2(\eta^{-1}f^2-\eta)\right) d \sigma \\
=&\int_{\Sigma_t}(\eta-f)\left(\Delta N  + N \bar{Ric} (\nu, \nu)+ \eta^{-1} f \Hb \frac{\partial N}{\partial \nu}\right) d \sigma \\
&+\int_{\Sigma_t}N\sigma_2\left(2(f-\eta^{-1}f^2)+\eta^{-1}f^2-\eta\right) d \sigma\\
=&\int_{\Sigma_t}(\eta-f)\left(\Delta N  + N \bar{Ric} (\nu, \nu)+ \eta^{-1} f \Hb \frac{\partial N}{\partial \nu}\right) d \sigma \\
&-\int_{\Sigma_t}N\sigma_2\eta^{-1}(\eta-f)^2d \sigma .
\end{align*}
The static equation \eqref{static equation} implies 
\begin{align*}
\Delta N+N\bar{Ric}(\nu, \nu)= \bar{\Delta}N-\bar{D}^2N(\nu,\nu)- \Hb \frac{\partial N}{\partial \nu}+ N \bar{Ric}(\nu, \nu)
= - \Hb \frac{\partial N}{\partial \nu} .
\end{align*}
Therefore, we conclude 
\begin{align*}
\frac{d}{dt}\left(\int_{\Sigma_t} N( \Hb  - H_\eta ) d \sigma\right) =&\int_{\Sigma_t}(\eta-f)(-1+ \eta^{-1}f) \Hb \frac{\partial N}{\partial \nu} d \sigma 
- \int_{\Sigma_t}N\sigma_2\eta^{-1}(\eta-f)^2d \sigma\\
=&-\int_{\Sigma_t}\eta^{-1}(\eta-f)^2 \Hb \frac{\partial N}{\partial \nu} d \sigma -\int_{\Sigma_t}N\sigma_2\eta^{-1}(\eta-f)^2d \sigma .
\end{align*}
\end{proof}

\begin{coro}\label{cor-1}
Suppose $(\N, \bg)$ has a positive static potential $N$. Along $ \{ \Sigma_t \}$, suppose 
\be \label{eq-monotone}
 \frac{\partial N}{\partial \nu} >  0 \ \mathrm{and} \  \sigma_i >  0 , \ i = 1, 2. 
\ee
Then  $ \displaystyle \int_{\Sigma_t} N ( \Hb - H_\eta) d \sigma $ is monotone nonincreasing 
and it is a constant if and only if $ \eta = f $.
\end{coro}

\section{Inverse curvature flows  in Schwarzschild manifolds} \label{sec-flow}

Corollary \ref{cor-1} suggests one consider foliations $\{ \Sigma_t \}$ satisfying condition \eqref{eq-monotone} in a static manifold 
with a positive static potential.  In this section, we use an inverse curvature flow to construct such foliations  in the 
Schwarzschild manifold $\SM$. 

We begin by fixing  some notations.  Henceforth, we will always use $\bg $ to  denote the metric on $ \SM$. 
We write 
\begin{align} \label{equation SM}
(\SM, \bar{g}) =( [ 0 , \infty) \times \mathbb{S}^n, dr^2+\phi^2(r)\sigma ), 
\end{align}
where $\sigma$ is the standard metric on the unit $n$-sphere $\mathbb{S}^n$ and $\phi = \phi(r) > 0 $ satisfies
$ \phi(0) = \left( 2 m \right)^\frac{1}{n-1} $ and 
\begin{align} \label{equation phi}
\phi^\prime=\sqrt{1 - 2m \phi^{1-n}}. 
\end{align}
In terms of this coordinate $r$,  the static potential function $N$  in Theorem \ref{thm-intro-main} equals $ \phi'$.  
We  use  $ \bar R(\cdot, \cdot, \cdot, \cdot)$, $ \bRic (\cdot, \cdot)$ to denote   the curvature tensor, 
the Ricci curvature of $\bg$, respectively. 
The  scalar curvature $\bar R$ of $ \bg$ is identically zero. 
 
Given any integer $ 1 \le k \le n $,  the Garding's  cone $\Gamma_k \subset \mathbb{R}^n$ is defined by  
$$\Gamma_k=\{(\kappa_1, \ldots, \kappa_n )\in \mathbb{R}^n \ | \ \sigma_j >0, \ 1  \leq j\leq k\} , $$
where $\sigma_j$ is the $j$-th elementary symmetric function of $(\kappa_1, \ldots, \kappa_n)$.
We also define $\sigma_0 = 1 $. A hypersurface $\Sigma \subset \SM$ is called $k$-convex if 
its principal curvature $(\kappa_1, \ldots, \kappa_n) \in \Gamma_k$. 

\begin{theo}\label{thm1}
Let $\Sigma^n_0$ be a star-shaped, $k$-convex, closed hypersurface in  $\SM$. Consider 
a smooth family of hypersurfaces $\{ \Sigma_t \}_{ t \ge0 }$ 
evolving according to  
\begin{align} \label{eq-k-flow}
\frac{\p X}{\p t} {}=\frac{\nu}{F},
\end{align}
where $\nu$ is the outward unit normal and $F=n\frac{C^{k-1}_n}{C^k_n}\frac{\sigma_k}{\sigma_{k-1}} > 0$
 which is evaluated at the principal curvatures of $\Sigma_t$. 
 Then  \eqref{eq-k-flow} has a smooth solution that exists for all time, each $ \Sigma_t$ remains star-shaped, 
 and the second fundamental form $ h$ of $\Sigma_t$ satisfies
\begin{align*}
|h^i_j\phi-\delta^i_j|\leq Ce^{-\alpha t},
\end{align*}
where $\phi$ is evaluated at $\Sigma_t$ and $C,\alpha$ depends only on $\Sigma_0,n,k$.
\end{theo}

We remark that 
inverse curvature flows in Euclidean spaces were first studied by Gerhardt \cite{G1} and Urbas \cite{U}.  They considered the  flow equation \eqref{eq-k-flow}
where $F$ is a concave function of homogeneous degree one, evaluated at the principal curvature, and  proved that the solution exists for all time and the normalized flow converges to a round sphere if the initial hypersurface is suitably star-shaped. 
For flows in other space forms, Gerhardt \cite{G2,G3} proved the solution exists for all time and the second fundamental form converges 
(see also earlier work by Ding \cite{D}). Recently, Brendle-Hung-Wang \cite{BHW} and Scheuer \cite{S2} proved that the same results hold in anti-de Sitter-Schwarzschild manifold and a class of warped product manifolds  for the inverse mean curvature flow, i.e. $F = \sigma_1$.
However, as pointed out by Neves \cite{N} and Hung-Wang \cite{HW}, for the inverse mean curvature flow, the rescaled limiting hypersurface is not necessarily 
a round sphere in an anti-de Sitter-Schwarzschild manifold.
The case of $F=n\frac{C^{k-1}_n}{C^k_n}\frac{\sigma_k}{\sigma_{k-1}}$  in anti-de Sitter-Schwarzschild manifolds 
was analyzed   by Lu \cite{Lu} and Chen-Mao \cite{CM} independently. They proved that 
the flow  exists for all time and the second fundamental converges exponentially fast if the initial hypersurface is 
star-shaped and $k$-convex.  

\medskip 

In what follows, we prove Theorem \ref{thm1} following the steps in \cite{Lu}.
We divide the proof into a few subsections. 

\subsection{Basic formulae}
We collect some well-known formulae  in Schwarzschild manifold in this subsection. 
Given a hypersurface $\Sigma^n \subset \SM$, we always use  $g$ to denote the induced metric on $\Sigma$.
Define 
\begin{align*}
\Phi(r)=\int_0^r\phi(\rho)d\rho, \quad u=\left\langle \phi \frac{\partial}{\partial r},\nu\right\rangle, 
\end{align*}
where $\nu$ is the outer unit normal of $\Sigma$ and $\langle \cdot, \cdot \rangle$ also denotes the metric product on $\SM$.
Let $i, j .. \in \{ 1 , \dots, n \}$ denote indices of local coordinates on $\Sigma$. 
Let $h$ be the second fundamental form on$ \Sigma$. 

The following formula is well-known (see  \cite{GL2} for instance), 
\begin{equation} \label{eq-Hess-Phi}
\Phi_{;ij} =\phi^\prime g_{ij}-h_{ij}u,
\end{equation}
where $``  ; " $ denotes the covariant differentiation on $\Sigma$.

Let $ R(\cdot, \cdot, \cdot, \cdot)$ be the curvature tensor of $ g$ on $\Sigma$. 
The Gauss equation and Codazzi equation are
\begin{align}\label{Gauss}
R_{ijkl}=\bar{R}_{ijkl}+\left(h_{ik}h_{jl}-h_{il}h_{jk}\right)
\end{align}
\begin{align}\label{Codazzi}
\nabla_k h_{ij}-\nabla_j h_{ik}=\bar{R}_{\nu ijk},
\end{align}
and the interchanging formula is
\begin{align}\label{interchanging formula}
\nabla_i\nabla_jh_{kl}=&\nabla_k\nabla_lh_{ij}-h^p_{l}(h_{ip}h_{kj}-h_{ij}h_{pk})-h^p_{j}(h_{pi}h_{kl}-h_{il}h_{pk}) \\\nonumber
&+h^p_{l}\bar{R}_{ikjp}+h^p_{j}\bar{R}_{iklp}+\nabla_k\bar{R}_{ijl\nu}+\nabla_i\bar{R}_{jkl\nu}.
\end{align}
Here $ \nabla$ is another notation for  the covariant differentiation on $\Sigma$.

The function $u$ is known as  the support function of $\Sigma$.  We have (see in \cite{Lu})
\begin{lemm}\label{support function}
\begin{align*}
\nabla_i u &=g^{kl}h_{ik}\nabla_l \Phi,\\
\nabla_i\nabla_j u&=g^{kl}\nabla_k h_{ij}\nabla_l\Phi+\phi^\prime h_{ij}-(h^2)_{ij}u+g^{kl}\nabla_l\Phi\bar{R}_{\nu jki},
\end{align*}
where $(h^2)_{ij}=g^{kl}h_{ik}h_{jl}$, $\bar{R}_{\nu jki}$ is the curvature of ambient space.
\end{lemm}

As for the curvature, we have the following curvature estimates, for proof, we refer readers to \cite{BHW}.

\begin{lemm}\label{ADS curvature}
The sectional curvature satisfies 
\begin{align*}
&\bar{R}(\partial_i,\partial_j,\partial_k,\partial_l)=\phi^2\left(1-{\phi^\prime}^2\right)(\sigma_{ik}\sigma_{jl}-\sigma_{il}\sigma_{jk})\\
&\bar{R}(\partial_i,\partial_r,\partial_j,\partial_r)=-\phi\phi^{\prime\prime}\sigma_{ij} .
\end{align*}
Together with \eqref{equation phi}, this gives 
\begin{align*}
&\bar{R}(\partial_i,\partial_j,\partial_k,\partial_l)= 2m\phi^{3-n}(\sigma_{ik}\sigma_{jl}-\sigma_{il}\sigma_{jk})\\
&\bar{R}(\partial_i,\partial_r,\partial_j,\partial_r)=- {m(n-1)} \phi^{1-n}\sigma_{ij} ,
\end{align*}
thus
\begin{align*}
\bar{R}_{\alpha\beta\gamma\mu}=O(r^{-n-1}),\quad \bar{\nabla}_\rho \bar{R}_{\alpha\beta\gamma\mu}=O(r^{-n-1}) .
\end{align*}
Here $ \{ \partial_i \}$ is the coordinate  frame on $\mathbb{S}^n$, $\sigma_{ij}$ is the standard metric of $\mathbb{S}^n$, and $\{ e_\alpha \} $ 
denotes  an orthonormal frame on $\SM$.
\end{lemm}

We also need the following two lemmas regarding to $\sigma_k$, see in \cite{Lu} for detailed proof.
\begin{lemm}\label{sigma_k 1}
let $F=n\frac{C^{k-1}_n}{C^k_n}\frac{ \sigma_k}{\sigma_{k-1}}$, thus $F$ is of homogeneous degree $1$, and $F(I)=n$, then we have
\begin{align*}
\sum_i F^{ii}\lambda_i^2\geq \frac{F^2}{n}
\end{align*}
\end{lemm}

\begin{lemm}\label{sigma_k 2}
Let $F=n\frac{C^{k-1}_n}{C^k_n}\frac{\sigma_k}{\sigma_{k-1}}$ and $(\lambda_i)\in \Gamma_k$, then 
\begin{align*}
n\leq \sum_i F^{ii}\leq nk
\end{align*}
\end{lemm}

\subsection{Parametrization on graph and $C^0$ estimate}

Since the initial hypersurface $\Sigma_0$ is star-shaped, we can consider it as a graph 
on $\mathbb{S}^n$, i.e. $X=(x,r)$ where $x$ is the coordinate on $\mathbb{S}^n$ and $r$ is the radial function.  
By taking derivatives, we have
\begin{align}\label{e_i}
X_i=\partial_i+r_i\partial_r,\quad g_{ij}=r_ir_j+\phi^2\sigma_{ij}
\end{align}
and
\begin{align}\label{v}
\nu=\frac{1}{v}\left (-\frac{r^i}{\phi^2}\partial_i+\partial_r\right ) , 
\end{align}
where $\nu$ is the unit normal vector, $v=(1+\frac{|\nabla r|^2}{\phi^2})^{\frac{1}{2}}$. 
Note that all the derivatives are on $\mathbb{S}^n$.

Thus
\begin{align*}
\frac{dr}{dt}=\frac{1}{Fv}, \dot{x}^i=-\frac{r^i}{\phi^2Fv}
\end{align*}
we have
\begin{align}\label{4.1}
\frac{\partial r}{\partial t}=\frac{dr}{dt}-r_j \dot{x}^j=\frac{v}{F}
\end{align}

By a direct computation, c.f. (2.6) in \cite{D} we have
\begin{align}\label{4.2}
h_{ij}=\frac{1}{v}(-r_{ij}+\phi\phi^\prime\sigma_{ij}+\frac{2\phi^\prime r_ir_j}{\phi})
\end{align}

Now we consider a function 
\begin{align}
\varphi=\int_{r_0}^r\frac{1}{\phi}
\end{align}
thus
\begin{align}\label{4.3}
\varphi_i=\frac{r_i}{\phi}, \varphi_{ij}=\frac{r_{ij}}{\phi}-\frac{\phi^\prime r_ir_j}{\phi^2}.
\end{align}
If we write everything in terms of $\varphi$, we have
\begin{align}\label{derivative of varphi}
\frac{\partial \varphi}{\partial t}=\frac{v}{\phi F}
\end{align}
and
\begin{align}
v=(1+|D\varphi|^2)^{\frac{1}{2}},g_{ij}=\phi^2(\varphi_i\varphi_j+\sigma_{ij}), g^{ij}=\phi^{-2}\left(\sigma^{ij}-\frac{\varphi^i\varphi^j}{v^2}\right).
\end{align}

Moreover, 
\begin{align}\label{4.6}
h_{ij}&=\frac{\phi}{v}\left( \phi^\prime(\sigma_{ij}+\varphi_i\varphi_j)-\varphi_{ij}\right),\\\nonumber
h^i_j&=g^{ik}h_{kj}=\frac{\phi^\prime}{\phi v}\delta^i_j-\frac{1}{\phi v}\tilde{\sigma}^{ik}\varphi_{kj}
\end{align}
where $\tilde{\sigma}^{ij}=\sigma^{ij}-\frac{\varphi^i\varphi^j}{v^2}$.

\begin{lemm}\label{C^0 estimate}
Let $\bar{r}(t)=\sup_{\mathbb{S}^n}r(\cdot,t)$ and $\underline{r}(t)=\inf_{\mathbb{S}^n}r(\cdot,t )$, then we have
\begin{align}
&\phi(\bar{r}(t))\leq e^{t/n}\phi(\bar{r}(0))\\\nonumber
&\phi(\underline{r}(t))\geq e^{t/n}\phi(\underline{r}(0))
\end{align}
\end{lemm}
\begin{proof}
Recall that $\frac{\partial r}{\partial t}=\frac{v}{F}$, where $F$ is a normalized operator on $(h^i_j)$. At the point where the function $r(\cdot, t)$ attains its maximum, we have $\nabla r=0, (r_{ij})\leq 0$, from (\ref{4.3}), we deduce that $\nabla \varphi=0, (\varphi_{ij})\leq 0$ at the maximum point. From (\ref{4.6}), we have $(h^i_j)\geq \left(\frac{\phi^\prime}{\phi }\delta^i_j\right)$, where we may assume $(g_{ij})$ and $(h_{ij})$ is diagonalized if necessary. Since $F$ is homogeneous of degree $1$, and $F(1,\cdots,1)=n$, we have
\begin{align*}
v^2=1+|\nabla\varphi|^2=1, F(h^i_j)\geq \frac{\phi^\prime}{\phi }F(\delta^i_j)=\frac{n\phi^\prime}{\phi },
\end{align*}
thus
\begin{align*}
\frac{d}{dt}\bar{r}(t)\leq \frac{\phi(\bar{r}(t))}{n\phi^\prime (\bar{r}(t))}
\end{align*}
i.e.
\begin{align*}
\frac{d}{dt}\log\phi(\bar{r}(t))\leq \frac{1}{n}
\end{align*}
which yields to the first inequality. Similarly, we can prove the second inequality, thus we have the lemma.
\end{proof}

\subsection{Evolution equations and $C^1$ estimate}
Before we go on with the estimate, let's derive some evolution equations first. 
\begin{align}
\dot{g}_{ij}=\frac{2h_{ij}}{F},\quad \dot{\nu}=\frac{g^{ij}F_ie_j}{F^2}
\end{align}

\begin{align}\label{6.5}
\dot{h}^i_j=-\frac{1}{F}h^{i}_kh^k_{j}-\nabla^i\nabla_j\left(\frac{1}{F}\right)-\frac{1}{F}\bar{R}^i_{\nu j\nu}
\end{align}

Together with the interchanging formula (\ref{interchanging formula}), we have
\begin{align}\label{6.6}
\dot{h}^i_j=&-\frac{1}{F}h^i_kh^k_j+ \frac{F^{pq,rs}{h_{pq}}^ih_{rsj}}{F^2}  -\frac{2F^{pq}{h_{pq}}^iF^{rs}h_{rsj}}{F^3}-\frac{1}{F}\bar{R}^i_{\nu j\nu}\\\nonumber
&+\frac{g^{ki}F^{pq}}{F^2} \big( h_{kj,pq}-h^l_{q}(h_{kl}h_{pj}-h_{kj}h_{lp})-h^l_{j}(h_{lk}h_{pq}-h_{kq}h_{lp})\\\nonumber
&+h^l_{q}\bar{R}_{kpjl}+h^l_{j}\bar{R}_{kpql}+\nabla_p\bar{R}_{kjq\nu}+\nabla_k\bar{R}_{jpq\nu} \big)
\end{align}
where $F^{ij}=\frac{\partial F}{\partial h_{pq}}$ and $F^{pq,rs}=\frac{\partial^2 F}{\partial h_{pq}\partial h_{rs}}$.

We also need the evolution equation for support function $u=\left\langle \phi\frac{\partial}{\partial r},\nu\right\rangle$.

\begin{align}\label{support derivative}
\dot{u}&=\frac{\phi^\prime}{F}+\frac{\phi  g^{ij}F_ir_j}{F^2}
\end{align}

Now, we need to consider the curvature term. By Lemma \ref{ADS curvature}, (\ref{e_i}) and (\ref{v}), we have
\begin{align*}
\bar{R}_{\nu jnk}&=\frac{r_n \delta_{jk}}{v}\left(-\phi\phi^{\prime\prime}-(1-(\phi^\prime)^2)\right)+\frac{r_k \delta_{jn}}{v}\left(\phi\phi^{\prime\prime}+(1-{\phi^\prime}^2)\right)
\end{align*}
Note that $g^{pn}=\phi^{-2}\left(\sigma^{p n}-\frac{r^p r^n}{v^2\phi^2}\right)$, where $ r^p = g^{pq} r_q$.
Thus
\begin{align}\label{6.10}
g^{pn}\nabla_p \Phi\bar{R}_{\nu jnk}&=\left(\frac{|\nabla r|^2\delta_{jk}-r_jr_k}{\phi v^3}\right)\left(-\phi\phi^{\prime\prime}-(1-{\phi^\prime}^2)\right)\leq 0 .
\end{align}

\begin{lemm}\label{C^1 estimate}
Along the flow, $|\nabla \varphi|\leq C$, where $C$ depends on $\Sigma_0,n,k$. In addition if $F^{ij}$ is uniformly elliptic and $\phi F$ is bounded above, then $|\nabla \varphi|\leq Ce^{-\alpha t}$, where $C,\alpha$ depends on $\Sigma_0,n,k$, uniform ellipticity constant of $F^{ij}$ and the upper bound of $\phi F$.
\end{lemm}

\begin{proof}
By (\ref{derivative of varphi}) and (\ref{4.6}), we have
\begin{align*}
\frac{\partial \varphi}{\partial t}=\frac{v}{\phi F}=\frac{v^2}{\tilde{F}(\phi^\prime \delta_{ij}-\tilde{\sigma}^{ik}\varphi_{kj})}=\frac{1}{G}
\end{align*}
where $\tilde{F}=\phi vF$.

Let $G^{ij}=\frac{\partial G}{\partial \varphi_{ij}}$, $G^k=\frac{\partial G}{\partial \varphi_k}$, $G_\varphi=\frac{\partial G}{\partial \varphi}$ then
\begin{align*}
G^{ij}=-\frac{1}{v^2}\tilde{F}^{i}_l\tilde{\sigma}^{lj},\quad G_\varphi=\frac{1}{v^2}\tilde{F}^i_i\phi\phi^{\prime\prime}
\end{align*}

Let $\omega=\frac{1}{2}|\nabla \varphi|^2$, we have
\begin{align*}
\frac{\partial \omega}{\partial t}&=\nabla_k\varphi\nabla_k\dot{\varphi}=-\frac{\varphi^k}{G^2}\nabla_kG=-\frac{\varphi^k}{G^2}\left(G^{ij}\varphi_{ijk}+G^l\varphi_{lk}+G_\varphi\varphi_k\right)\\
&=\frac{1}{v^2G^2}\left(\tilde{F}^{i}_l\tilde{\sigma}^{lj}\varphi^k\varphi_{ijk}-v^2 G^l\omega_l-2\tilde{F}^{i}_i\phi\phi^{\prime\prime} \omega\right)
\end{align*}

We want to write the term $\tilde{\sigma}^{lj}\varphi_{ijk}$ in terms of second derivative of $\omega$. Note that
\begin{align*}
\omega_{ij}&=\varphi_{kij}\varphi^k+\varphi_{ki}\varphi^{k}_j\\
&=\varphi_{ijk}\varphi^k+(\sigma_{ij}\sigma_{kp}-\sigma_{ik}\sigma_{jp})\varphi^p\varphi^k+\varphi_{ki}\varphi^{k}_j\\
&=\varphi_{ijk}\varphi^k+\sigma_{ij}|\nabla\varphi |^2-\varphi_i\varphi_j+\varphi_{ki}\varphi^{k}_j
\end{align*}
and
\begin{align*}
\tilde{\sigma}^{lj}\left(\sigma_{ij}|\nabla\varphi |^2-\varphi_i\varphi_j\right)=\delta_{i}^l|\nabla \varphi|^2-\varphi_i\varphi^l
\end{align*}
Thus we have
\begin{align}
\frac{\partial w}{\partial t}=\frac{1}{v^2G^2}\left(\tilde{F}^{i}_l\tilde{\sigma}^{lj}\omega_{ij}-\tilde{F}^{i}_i|\nabla \varphi|^2+\tilde{F}^{i}_l\varphi_i\varphi^l-v^2 G^l\omega_l-2\tilde{F}^{i}_i\phi\phi^{\prime\prime} \omega\right)-\frac{1}{v^2G^2}\tilde{F}^{i}_l\tilde{\sigma}^{lj}\varphi_{ki}\varphi^{k}_j
\end{align}

Note that $-\tilde{F}^{i}_i|\nabla \varphi|^2+\tilde{F}^{i}_l\varphi_i\varphi^l\leq 0$ and $-\tilde{F}^{i}_l\tilde{\sigma}^{lj}\varphi_{ki}\varphi^{k}_j\leq 0$, thus by the maximum principle, we have
\begin{align*}
\omega(\cdot,t)\leq \sup\omega_0.
\end{align*}

Now if $F^{ij}$ is uniformly elliptic, i.e. $\tilde{F}^{ij}$ is uniformly elliptic and $\phi F$ is bounded above, then consider $\tilde{\omega}=\omega e^{\lambda t}$, at the maximum point, we have
\begin{align*}
\frac{\partial \tilde{\omega}}{\partial t} \leq \frac{1}{v^2G^2}\left(-\tilde{F}^{i}_i|\nabla \varphi|^2+\tilde{F}^{i}_l\varphi_i\varphi^l\right)e^{\lambda t}+\lambda\tilde{\omega}\leq (-\frac{c}{\phi^2F^2}+\lambda)\tilde{\omega}
\end{align*}
thus $\tilde{\omega}$ is uniformly bounded, we have $|\nabla\varphi|$ decays exponentially.
\end{proof}

\begin{lemm}\label{lem-Ric}
Suppose $\bRic(\nu,\nu)\leq 0$ for $\Sigma_0$, then $\bRic(\nu,\nu)\leq 0$ for all $\Sigma_t$. 
If $ k \ge 2$, this implies $R>0$ for all $\Sigma_t$. 
\end{lemm}

\begin{proof}
By Lemma \ref{ADS curvature}, we have
\begin{align*}
\bRic =\left((n-1)(1-{\phi^\prime}^2)-\phi\phi^{\prime\prime}\right)g_{\mathbb{S}^n}-n\frac{\phi^{\prime\prime}}{\phi}dr^2
\end{align*}
Together with (\ref{v}), i.e. $\nu=\frac{1}{v}\left(\partial_r-\frac{r^j\partial_j}{\phi^2}\right)$, we have
\begin{align*}
\bRic (\nu,\nu)&=-n\frac{\phi^{\prime\prime}}{\phi v^2}+\frac{(n-1)(1-{\phi^\prime}^2)-\phi\phi^{\prime\prime}}{\phi^4 v^2}|\nabla r|^2\\
&=-n\frac{\phi^{\prime\prime}}{\phi v^2}+\frac{(n-1)(1-{\phi^\prime}^2)-\phi\phi^{\prime\prime}}{\phi^2 v^2}(v^2-1)\\
&=\frac{(n-1)(1-{\phi^\prime}^2)-\phi\phi^{\prime\prime}}{\phi^2}-(n-1)\frac{1-{\phi^\prime}^2+\phi\phi^{\prime\prime}}{\phi^2 v^2} .
\end{align*}
Since $\phi^\prime=\sqrt{1- 2m\phi^{1-n}}$, thus
\begin{align*}
1-{\phi^\prime}^2= 2m\phi^{1-n},\quad \phi\phi^{\prime\prime}= {m(n-1)} \phi^{1-n} .
\end{align*}
Thus
\begin{align*}
\bRic (\nu,\nu)={m(n-1)} \phi^{-1-n}-  {m(n-1)(n+1)} \phi^{-1-n} v^{-2} . 
\end{align*}

On the other hand, $v^2=1+|\nabla\varphi|^2$ and, by Lemma \ref{C^1 estimate}, $|\nabla\varphi|$ is bounded above by the initial data. 
Thus it follows that, if initially $\bRic(\nu,\nu)\leq 0$, i.e. $|\nabla\varphi|^2\leq n$,  then it remains true along the flow.  

To prove the second assertion, it suffices to note that
\begin{align*}
\sigma_2=\frac{R}{2}+\bRic (\nu,\nu)>0
\end{align*}
along the flow. Thus $R>0$ along the flow.
\end{proof}

\subsection{Bound for principal curvature}
\begin{lemm}\label{lemma 8}
Along the flow, $F\phi\leq C$, where $C$ depends only on $\Sigma_0,n,k$. In addition, if $F^{ij}$ is uniformly elliptic, then $F\phi\leq n+Ce^{-\alpha t}$, where $C,\alpha$ depends only on $\Sigma_0,n,k$ and the uniform ellipticity constant of $F^{ij}$.
\end{lemm}

\begin{proof}
Consider $ F \phi$, at the maximum point, we have
\begin{align*}
\frac{\dot{\phi}}{\phi}+\frac{\dot{F}}{F}\geq 0
\end{align*}
and
\begin{align*}
\frac{\phi_i}{\phi}+\frac{F_i}{F}=0,\quad \frac{\phi_{ij}}{\phi}+\frac{F_{ij}}{F}-2\frac{F_iF_j}{F^2}\leq 0
\end{align*}

By (\ref{4.1}) and (\ref{6.5}), we have
\begin{align*}
0&\leq \frac{\phi^\prime v}{F\phi}+\frac{F^j_i}{F} \left(-\frac{1}{F}h^{i}_kh^k_{j}-\nabla^i\nabla_j\left(\frac{1}{F}\right)-\frac{1}{F}\bar{R}^i_{\nu j\nu}\right)\\
&=\frac{\phi^\prime v}{F\phi}+\frac{F^j_i}{F}\left(-\frac{1}{F}h^{i}_kh^k_{j}+\frac{\nabla^i\nabla_jF}{F^2}-2\frac{\nabla^i F\nabla_jF}{F^3}-\frac{1}{F}\bar{R}^i_{\nu j\nu}\right)
\end{align*}

By the critical equation above and (\ref{4.2}), we have
\begin{align*}
0&\leq  \frac{\phi^\prime v}{F\phi}+\frac{F^j_i}{F}\left(-\frac{1}{F}h^{i}_kh^k_{j}-\frac{\nabla^i\nabla_j \phi}{F\phi}-\frac{1}{F}\bar{R}^i_{\nu j\nu}\right)\\
&=\frac{\phi^\prime v}{F\phi}+\frac{F^j_i}{F^2}\left(-h^{i}_kh^k_{j}-\bar{R}^i_{\nu j\nu}\right)-\frac{F^{ij}}{F^2\phi}\left(\phi^{\prime\prime}r_ir_j+\phi^\prime r_{ij}\right)\\
&=\frac{\phi^\prime v}{F\phi}+\frac{F^j_i}{F^2}\left(-h^{i}_kh^k_{j}-\bar{R}^i_{\nu j\nu}\right)-\frac{F^{ij}}{F^2\phi}\left( \phi^{\prime\prime}r_ir_j+\phi^\prime\left(\phi\phi^\prime\sigma_{ij}+\frac{2\phi^\prime r_ir_j}{\phi}-h_{ij}v\right)\right)\\
&=2\frac{\phi^\prime v}{F\phi}-\frac{F^j_i}{F^2}\left(h^{i}_kh^k_{j}+\bar{R}^i_{\nu j\nu}\right)-\frac{F^{ij}}{F^2\phi}\left( \phi^{\prime\prime}r_ir_j+\phi^\prime\left(\phi\phi^\prime\sigma_{ij}+\frac{2\phi^\prime r_ir_j}{\phi}\right)\right)
\end{align*}

By lemma \ref{ADS curvature}, Lemma \ref{sigma_k 1}, Lemma \ref{sigma_k 2}, Lemma \ref{C^1 estimate} and property of $\phi$, we have
\begin{align*}
0&\leq \frac{2v}{F\phi}-\frac{1}{n}+C\frac{F^i_i }{F^2\phi^{n+1}}-\frac{F^i_i}{F^2\phi^2}\\
&\leq \frac{C}{F\phi}-\frac{1}{n}+C\frac{F^i_i }{F^2\phi^{n+1}}-\frac{n}{F^2\phi^2}\\
&\leq \frac{C}{F\phi}-\frac{1}{n}+C\frac{F^i_i }{F^2\phi^{n+1}}
\end{align*}
thus $F\phi$ is bounded above.

If in addition $F^{ij}$ is uniformly elliptic, by Lemma \ref{C^1 estimate}, $|\nabla\varphi|$ decays exponentially, then
\begin{align*}
0\leq \frac{2}{F\phi}-\frac{1}{n}-\frac{n}{F^2\phi^2}+Ce^{-\alpha t}
\end{align*}
i.e.
\begin{align*}
F\phi\leq n+Ce^{-\alpha t}.
\end{align*}
\end{proof}

\begin{lemm}\label{time derivative}
Along the flow, $|\dot{\varphi}|\leq C$, where $C$ depends on $\Sigma_0, n, k$.
\end{lemm}
\begin{proof}
By (\ref{derivative of varphi}) and (\ref{4.6}), we have
\begin{align*}
\frac{\partial \varphi}{\partial t}=\frac{v}{\phi F}=\frac{v^2}{\tilde{F}(\phi^\prime \delta_{ij}-\tilde{\sigma}^{ik}\varphi_{kj})}=\frac{1}{G}
\end{align*}
where $\tilde{F}=\phi vF$.

Let $G^{ij}=\frac{\partial G}{\partial \varphi_{ij}}$, $G^k=\frac{\partial G}{\partial \varphi_k}$, $G_\varphi=\frac{\partial G}{\partial \varphi}$ then
\begin{align*}
G^{ij}=-\frac{1}{v^2}\tilde{F}^{i}_l\tilde{\sigma}^{lj},\quad G_\varphi=\frac{1}{v^2}\tilde{F}^i_i\phi\phi^{\prime\prime}
\end{align*}
thus
\begin{align*}
\frac{\partial \dot{\varphi}}{\partial t}&=-\frac{\dot{G}}{G^2}=-\frac{1}{G^2}\left(G^{ij}\dot{\varphi}_{ij}+G^k\dot{\varphi_k}+G_\varphi\dot{\varphi}\right)\\
&=\frac{1}{v^2G^2}\left(\tilde{F}^{i}_l\tilde{\sigma}^{lj}\dot{\varphi}_{ij}-v^2 G^k\dot{\varphi}_k-\tilde{F}^{i}_i\phi\phi^{\prime\prime} \dot{\varphi}\right)
\end{align*}
By maximum principle, we conclude that $|\dot{\varphi}|$ is bounded above.
\end{proof}

\begin{lemm}\label{lemma 9}
Along the flow, $F\phi\geq c$, where $c$ depends on $\Sigma_0,n,k$.
\end{lemm}

\begin{proof}
Since $\dot{\varphi}=\frac{v}{\phi F}$, by Lemma \ref{time derivative}, we have
\begin{align*}
\frac{v}{\phi F}\leq C
\end{align*}
thus $F\phi\geq c$.
\end{proof}

\begin{lemm}\label{princial curvature}
Along the flow, $|\kappa_i\phi|\leq C$, where $\kappa_i$ is the principal curvature of $\Sigma_t$, $C$ depends on $\Sigma_0,n,k$.
\end{lemm}

\begin{proof}
Consider $\log\eta-\log u+\frac{2t}{n}$, where
\begin{align*}
\eta=\sup\{h_{ij}\xi^i\xi^j:g_{ij}\xi^i\xi^j=1\}
\end{align*}
WLOG, we suppose that at the maximum point $\eta=h^1_1$, and we have
\begin{align}\label{9.1}
\frac{\dot{h_1^1}}{h^1_1}-\frac{\dot{u}}{u}+\frac{2}{n}\geq 0 
\end{align}
and
\begin{align}\label{9.2}
\frac{h^1_{1i}}{h^1_1}-\frac{u_i}{u}=0,\quad\frac{h^1_{1ij}}{h^1_1}\leq \frac{u_{ij}}{u}
\end{align}
by (\ref{6.6}), (\ref{support derivative}) and the critical equation, we have
\begin{align}\label{9.3}
0\leq& \frac{1}{h^1_1}\bigg(-\frac{1}{F}h^1_kh^k_1+ \frac{F^{pq,rs}{h_{pq}}^1h_{rs1}}{F^2}  -\frac{2F^{pq}{h_{pq}}^1F^{rs}h_{rs1}}{F^3}-\frac{1}{F}\bar{R}^1_{\nu 1\nu}\\\nonumber
&+\frac{g^{k1}F^{pq}}{F^2} \big( h_{k1,pq}-h^m_{q}(h_{km}h_{p1}-h_{k1}h_{mp})-h^m_{1}(h_{mk}h_{pq}-h_{kq}h_{mp})\\\nonumber
&+h^m_{q}\bar{R}_{kp1m}+h^m_{1}\bar{R}_{kpqm}+\nabla_p\bar{R}_{k1q\nu}+\nabla_k\bar{R}_{1pq\nu} \big)\bigg)\\\nonumber
&-\frac{\phi^\prime}{Fu}-\frac{\phi  g^{ij}F_ir_j}{F^2u}+\frac{2}{n}
\end{align}
consider the term $\frac{F^{pq}}{F^2} \frac{h^1_{1,pq}}{h^1_1} $, by (\ref{9.2}) and lemma \ref{support function},  we have
\begin{align}\label{9.4}
\frac{F^{pq}}{F^2} \frac{h^1_{1,pq}}{h^1_1} \leq \frac{F^{pq}}{F^2}\frac{u_{pq}}{u}=\frac{ F^{pq}}{F^2u}\left(g^{kl}h_{pqk}\Phi_l+\phi^\prime h_{pq}-(h^2)_{pq}u+g^{kl}\nabla_l\Phi\bar{R}_{\nu pkq}\right)
\end{align}
insert (\ref{9.4}) into (\ref{9.3}), together with the concavity of $F$, yields
\begin{align}
0\leq& \frac{1}{h^1_1}\bigg(-\frac{1}{F}h^1_kh^k_1-\frac{1}{F}\bar{R}^1_{\nu 1\nu}+\frac{g^{k1}F^{pq}}{F^2} \big(-h^m_{1}h_{mk}h_{pq}+h^m_{q}\bar{R}_{kp1m}+h^m_{1}\bar{R}_{kpqm}+\nabla_p\bar{R}_{k1q\nu}+\nabla_k\bar{R}_{1pq\nu} \big)\bigg)\\\nonumber
&+\frac{ g^{kl}F^{pq}}{F^2u}\nabla_l\Phi\bar{R}_{\nu pkq}+\frac{2}{n}
\end{align}

By (\ref{6.10}), we have
\begin{align}\label{9.6}
0&\leq \frac{1}{h^1_1}\bigg(-\frac{2}{F}h^1_kh^k_1-\frac{1}{F}\bar{R}^1_{\nu 1\nu}+\frac{g^{k1}F^{pq}}{F^2} \big(h^m_{q}\bar{R}_{kp1m}+h^m_{1}\bar{R}_{kpqm}+\nabla_p\bar{R}_{k1q\nu}+\nabla_k\bar{R}_{1pq\nu} \big)\bigg)+\frac{2}{n}
\end{align}

By Lemma \ref{ADS curvature}, all terms involving curvature terms of the ambient space are uniformly bounded by $C\phi^{-1-n}$, i.e.
\begin{align*}
\frac{g^{k1}F^{pq}}{F^2} \big(h^m_{q}\bar{R}_{kp1m}+h^m_{1}\bar{R}_{kpqm}+\nabla_p\bar{R}_{k1q\nu}+\nabla_k\bar{R}_{1pq\nu} \big)\leq \frac{CF^i_i}{F^2\phi^{n+1}} h^1_1 \leq Ch^1_1
\end{align*}
we have used Lemma \ref{sigma_k 2} and Lemma \ref{lemma 9} in the last inequality.

Plug into (\ref{9.6}), we have
\begin{align*}
0&\leq \frac{1}{h^1_1}\left(-\frac{2}{F}h^1_kh^k_1-\frac{1}{F}\bar{R}^1_{\nu 1\nu}\right)+C\leq \frac{1}{h^1_1}\left(-\frac{2}{F\phi}h^1_kh^k_1\phi+\frac{C}{F\phi^{n+1}}\right)+C\\
&\leq -Ch^1_1\phi+\frac{C}{h^1_1\phi}+C
\end{align*}
i.e. $h^1_1\phi\leq C$, since $ce^{\frac{t}{n}}\leq u\leq \phi\leq Ce^{\frac{t}{n}}$, we have the lemma.
\end{proof}

\subsection{Asymptotic behaviors}

\begin{lemm}\label{asymptotic second fundamental form}
$|h^i_j\phi-\delta^i_j|\leq Ce^{-\alpha t}$ where $C,\alpha$ depends only on $\Sigma_0,n,k$.  Moreover for any $p,q\geq 0$, we have $|\left(\frac{\partial}{\partial t}\right)^p\left(\phi\nabla\right)^q\phi^2\nabla h^i_j|\leq Ce^{-\alpha t}$ , where $\nabla$ is the unit gradient on $\Sigma_t$ and $C$ depends in addition on $p,q$.
\end{lemm}

\begin{proof}
To prove the lemma, we first notice that by (\ref{4.6}) and (\ref{derivative of varphi}), we have
\begin{align*}
h^i_j=\frac{\phi^\prime}{\phi v}\delta^i_j-\frac{1}{\phi v}\tilde{\sigma}^{ik}\varphi_{kj}
\end{align*}
and
\begin{align*}
\frac{\partial \varphi}{\partial t}=\frac{v}{\phi F}=\frac{v}{\tilde{F}}
\end{align*}
where 
\begin{align*}
\tilde{F}=\phi F=F(\frac{\phi^\prime}{v}\delta^i_j-\frac{1}{ v}\tilde{\sigma}^{ik}\varphi_{kj})
\end{align*}
By the Lemma \ref{C^1 estimate} and Lemma \ref{princial curvature}, we know that $\nabla\varphi$ and $\nabla^2\varphi$ is uniformly bounded. By Evans-Krylov, we have $|\varphi|_{2,\alpha}\leq C$. By standard interpolation inequality, we have $\nabla^2\varphi$ decays exponentially as $\nabla \varphi$ decays exponentially. Thus from the definition of $h^i_j$ above, we have the first inequality.

By Schauder estimate, we have $|\varphi|_l\leq Ce^{-\alpha t}$ for all $l\geq 1$.

By the definition of $h^i_j$, we have
\begin{align*}
\nabla h^i_j=&\left(\frac{\phi^{\prime\prime}}{\phi v}-\frac{{\phi^\prime}^2}{\phi^2 v}\right)\delta^i_j\nabla r-\frac{\phi^\prime}{\phi v^3}\delta^i_j\varphi_k\nabla \varphi_k\\
&+\frac{\phi^\prime}{\phi^2 v}\tilde{\sigma}^{ik}\varphi_{kj}\nabla r+\frac{1}{\phi v^3}\tilde{\sigma}^{ik}\varphi_{kj}\varphi_l\nabla \varphi_l\\
&+\frac{1}{\phi v}\nabla\varphi^i\varphi^k\varphi_{kj}+\frac{1}{\phi v}\nabla\varphi^k\varphi^i\varphi_{kj}-\frac{1}{\phi v}\tilde{\sigma}^{ik}\nabla\varphi_{kj}
\end{align*}
Since $|\varphi|_l\leq Ce^{-\alpha t}$ for all $l\geq 1$,  this implies 
\begin{align*}
|\phi^2\nabla h^i_j|\leq Ce^{-\alpha t} .
\end{align*}
By induction, we have
\begin{align*}
|\left(\frac{\partial}{\partial t}\right)^p\left(\phi\nabla\right)^q\phi^2\nabla h^i_j|\leq Ce^{-\alpha t}
\end{align*}
for all $p,q\geq 0$.
\end{proof}

\begin{lemm}\label{renormalized metric decay}
Let $\tilde{g}_{ij}=\phi^{-2}g_{ij}$ be a normalized metric, then $|\tilde{g}_{ij}-\sigma_{ij}|\leq Ce^{-\alpha t}$, where $\sigma_{ij}$ is the standard metric on $\mathbb{S}^n$ and $C,\alpha$ depends only on $\Sigma_0,n,k$.  Moreover for any $p,q\geq 0$, we have $|\left(\frac{\partial}{\partial t}\right)^p\left(\phi\nabla\right)^q\phi\nabla \tilde{g}_{ij}|\leq Ce^{-\alpha t}$ , where $\nabla$ is the unit gradient on $\Sigma_t$ and $C$ depends in addition on $p,q$.
\end{lemm}

\begin{proof}
Following the step in \cite{G1}, we consider the rescaled hypersurface as
$ \hat{X}=Xe^{-\frac{t}{n}}$
then we have
$ \hat{r}=re^{-\frac{t}{n}}$,
thus
\begin{align*}
\hat{g}_{ij}=\phi^2(\hat{r})\sigma_{ij}+\hat{r}_i\hat{r}_j
\end{align*}
By Lemma \ref{C^0 estimate} and Lemma \ref{C^1 estimate} , we have $c_0\leq \hat{r}\leq C_0$ uniformly, and $|\hat{r}_i|\leq Ce^{-\alpha t}$, thus 
\begin{align*}
c_0\sigma\leq \hat{g}\leq C_0\sigma
\end{align*}
for $t$ large enough, i.e. $\hat{g}$ is well defined.

Now let's prove that $\hat{g}$ converges to $\hat{g}_\infty$. By Lemma \ref{C^1 estimate}, we have
\begin{align*}
\frac{\partial \hat{g}_{ij}}{\partial t}=2\phi(\hat{r})\phi^\prime(\hat{r})\left(\frac{v}{F}e^{-\frac{t}{n}}-\frac{1}{n}re^{-\frac{t}{n}}\right)\sigma_{ij}+\frac{\partial }{\partial t}\left(\hat{r}_i\hat{r}_j\right)\leq Ce^{-\alpha t}
\end{align*}

Thus $\hat{g}$ converges exponentially fact to $\hat{g}_\infty$. To prove that $\hat{g}_\infty$ is a round metric, we only need to prove that $\hat{r}$ is constant. Since $\hat{r}$ is defined on $\mathbb{S}^n$, we take derivative of $\mathbb{S}^n$ on $\hat{r}$ to obtain
\begin{align*}
|\nabla_{\mathbb{S}^n}\hat{r}|=|\nabla_{\mathbb{S}^n}r e^{-\frac{t}{n}}|\leq Ce^{-\alpha t}
\end{align*}
Thus $\hat{r}$ is constant for $t=\infty$, i.e. we have
\begin{align*}
r=r_0e^{\frac{t}{n}}+O(e^{(\frac{1}{n}-\alpha) t})
\end{align*}
and
\begin{align*}
\phi(r)=r_0e^{\frac{t}{n}}+O(e^{(\frac{1}{n}-\alpha) t})
\end{align*}
Hence, at time $t$, we have
\begin{align*}
g_{ij}=\phi^2(r)\left(\sigma_{ij}+\varphi_i\varphi_j\right)=r_0^2e^{\frac{2t}{n}}\sigma_{ij}+O(e^{(\frac{2}{n}-2\alpha) t}) ,
\end{align*}
and the normalized metric $\tilde{g}_{ij}$ satisfies
\begin{align*}
\tilde{g}_{ij}=\phi^{-2}g_{ij}=\sigma_{ij}+O(e^{-\alpha t}). 
\end{align*}

Similar to the previous lemma, high regularity decay estimates follows by Lemma \ref{C^1 estimate} 
and the definition of $\tilde{g}_{ij}$.
\end{proof}

\begin{rema}
Let  $ k \ge 2$.
Let $ g$ be a metric on $\mathbb{S}^n$ so that $(S^n, g)$ isometrically embeds into $ \SM$ as  
a star-shaped, $k$-convex, closed hypersurface in $\SM$ with $ \Ric (\nu, \nu) \le 0$. 
Combining results in this section and arguments in \cite[Section 3]{CPM-hd}, one knows that 
$g$ can be connected to a round metric within the space of  positive scalar curvature metrics on $\mathbb{S}^n$. 
Therefore, repeating the proof in \cite{CPM-hd},  we know that 
the conclusion of \cite[Theorem 1.2]{CPM-hd} holds for such a metric $g$.
\end{rema}

\section{Bartnik-Shi-Tam type  asymptotically flat  extensions} \label{sec-warped-metric}

Let  $\Sigma^n \subset \SM $ be a closed, star-shaped, $2$-convex hypersurface satisfying 
\be \label{eq-initial-Ric}
\bRic (\nu, \nu) \le 0 .
\ee
Here  $ \bRic(\cdot, \cdot) $ is the Ricci curvature of the Schwarzschild manifold  $\SM$ and $\nu$ is the outward unit normal to $ \Sigma$. 
By  Theorem \ref{thm1}, there exists a smooth solution $ \{ \Sigma_t \}_{ 0 \le t \le \infty }$, consisting of star-shaped hypersurfaces,  to 
\begin{align} \label{eq-2-flow}
\frac{\p X}{\p t} {}=\frac{n-1}{2n}\frac{\sigma_1}{\sigma_2} \nu 
\end{align}
with initial condition $ \Sigma_0 = \Sigma$.   
By Lemma \ref{lem-Ric}, condition \eqref{eq-initial-Ric} implies that 
the scalar curvature $R$ of each $ \Sigma_t $ is positive. 

Let $ \mathbb{E}$ denote the exterior of $ \Sigma$ in $ \SM$, which is swept by $ \{ \Sigma_t \}_{ 0 \le t \le \infty }$.
On $ \E$,  the Schwarzschild metric $\bg $ can be written as 
\begin{align*}
\bar{g}=f^2dt^2+g_t ,
\end{align*}
where $g_t$ is the induced metric on $\Sigma_t $ and 
$$f =\frac{n-1}{2n}\frac{\sigma_1}{\sigma_2 } > 0 . $$ 
Prompted  by Proposition \ref{prop-mono}, 
we are interested in a new  metric $g_\eta$  on $\E$,  which
  takes the form of 
\begin{align*}
g_\eta=\eta^2dt^2+g_t  ,
\end{align*}
and has zero scalar curvature.
Here $ \eta > 0 $ is a function on $ \E$. 

We first derive the equation for $ \eta$. Adopting the notations in Section \ref{sec-mono-static}, 
by (\ref{formula H,H_eta}), (\ref{H_eta}) and Gauss equation (\ref{Gauss-1}), we have
\begin{align*}
\frac{\partial}{\partial t}  H_\eta = &- \Delta \eta - \eta ( | A_\eta |^2 + Ric_{g_\eta} (\nu, \nu)) \\
= &  - \Delta \eta- \eta \left(\eta^{-2}f^2| \Ab |^2+\eta^{-2}f^2\sigma_2-\frac{R}{2} \right)\\
= &  - \Delta \eta - \eta^{-1}f^2 |\Ab |^2-\eta^{-1}f^2\sigma_2+\frac{R}{2}\eta .
\end{align*}
On the other hand,
\begin{align*}
\frac{\partial}{\partial t}  H_\eta =\frac{\partial}{\partial t} (\frac{f \Hb }{\eta})
=-\frac{f \Hb }{\eta^2}\frac{\partial \eta}{\partial t}+\frac{1}{\eta}\frac{\partial}{\partial t}(f\Hb ) .
\end{align*}
Thus
\begin{align*}
-\frac{f \Hb }{\eta^2}\frac{\partial \eta}{\partial t}+\frac{1}{\eta}\frac{\partial}{\partial t}(f \Hb )
= - \Delta \eta - u^{-1}f^2| \Ab |^2-\eta^{-1}f^2\sigma_2+\frac{R}{2}\eta
\end{align*}
i.e.
\begin{align}\label{equation u}
-\frac{\partial \eta}{\partial t}+\frac{\eta^2}{f \Hb }\Delta \eta =\frac{\eta^3R}{2 f \Hb }- \frac{\eta}{ f \Hb }
\left(f^2| \Ab |^2+f^2\sigma_2+\frac{\partial}{\partial t}(f \Hb )\right) .
\end{align}

\medskip

Equation \eqref{equation u} is  as the same as (5) in \cite{EMW}. 
The following assertion on the long time existence of $\eta$ on $ \E $ follows directly from 
\cite[Proposition 2]{EMW} and  Lemma \ref{lem-Ric}. 

\begin{lemm}\label{C^0 perturbed}
Let $\Sigma$ be a closed, star-shaped, $2$-convex hypersurface in $ \SM$  with $\bRic (\nu,\nu)\leq 0$.
Given any positive function $ \psi > 0 $ on $ \Sigma$, 
the solution to   (\ref{equation u}) 
with initial condition $ \eta |_{t=0} = \psi $
exists for all time and remains positive.
\end{lemm}

In what follows, we analyze the asymptotic behavior of  $g_\eta $. 

\subsection{$C^0$ estimate of $\eta$}
For the convenience of  estimating  $\eta$, we consider   
$$w= f^{-1} \eta . $$ 
By \eqref{equation u}, \eqref{H} and \eqref{Gauss-1}, it is easily seen that 
$w$ satisfies the  equation
\begin{align}\label{equation w}
-\frac{\partial w}{\partial t} +\frac{w^2}{ \Hb }\left(f\Delta w+2\nabla w\nabla f\right)=\frac{1}{2 \Hb }\left(fR-2\Delta f\right)(w^3-w) .
\end{align}

\begin{lemm}\label{C^0 w}
$w$ satisfies the  estimate
\begin{align*}
|w-1|\leq C\phi^{1-n} ,
\end{align*}
where $C$ depends only on $\Sigma_0$ and $n$.
\end{lemm}

\begin{proof}
It suffices to focus on $ w$ for $ t \ge t_0$ where $ t_0$ is sufficiently large.
Following the steps in \cite{ST}, we define
\begin{align*}
A(t)=\min_{\Sigma_t}\frac{fR-2\Delta f}{ \Hb },\quad B(t)=\max_{\Sigma_t}\frac{fR-2\Delta f}{ \Hb } .
\end{align*}
By Lemma \ref{asymptotic second fundamental form}, Lemma \ref{ADS curvature} and Gauss equation \eqref{Gauss-1}, we have
\begin{align*}
\frac{fR-2\Delta f}{ \Hb }=\frac{n-1}{n}+Ce^{-\alpha t} ,
\end{align*}
thus both $A(t)$ and $B(t)$ are positive for $t\geq t_0$.

We first seek an  upper bound for $w$. Define
\begin{align*}
P(t)=\left(1-C_1\exp(-\int_{t_0}^tA(s)ds)\right)^{-\frac{1}{2}}
\end{align*}
with $C_1=1-(\max_{\Sigma_{t_0}} w+1)^{-2}$. It is  clear that $P-w\geq 0$ at $t_0$.
Taking derivative, we have
\begin{align*}
\frac{d}{dt}P(t)&=-\frac{1}{2}\left(1-C_1\exp(-\int_{t_0}^tA(s)ds)\right)^{-\frac{3}{2}}C_1\exp(-\int_{t_0}^tA(s)ds)A(t)\\
&=\frac{1}{2}P^3(P^{-2}-1)A=\frac{1}{2}( P- P^3)A .
\end{align*}
At the minimum point of $P-w$, we have
\begin{align*}
 \frac{d}{dt}\left(P-w\right)\leq 0,\quad \nabla w=0,\quad \nabla^2 w\leq 0 ,
\end{align*}
thus
\begin{align*}
0\geq \frac{1}{2}(P- P^3)A+\frac{1}{2 \Hb}\left(fR-2\Delta f\right)(w^3-w) .
\end{align*}
Since $A\leq \frac{fR-2\Delta f}{\Hb }$, we have 
\begin{align*}
0\geq P-P^3+w^3-w ,
\end{align*}
i.e. $ P-w\geq 0$ as $P\geq 1$.
Therefore,   $w\leq P$ for all time $t \ge t_0$. 

\medskip

Next, we  seek a lower bound of $w$. We consider  two cases.

\medskip

\noindent Case 1: $\min_{\Sigma_{t_0}} w\geq 1$. Define
\begin{align*}
Q(t)=\left(1+C_2\exp(-\int_{t_0}^tB(s)ds)\right)^{-\frac{1}{2}} ,
\end{align*}
where $C_2=(\min_{\Sigma_{t_0}} w)^{-2}-1$. It's clear that $w-Q\geq 0$ at $t_0$.
By a similar computation as above, we have
\begin{align*}
\frac{d}{dt}Q(t)=\frac{1}{2}(Q-Q^3)B .
\end{align*}
At the minimum point of $w-Q$, 
\begin{align*}
\frac{d}{dt}(w-Q)\leq 0,\quad \nabla w=0,\quad \nabla^2w\geq 0 .
\end{align*}
Thus
\begin{align*}
0\geq -\frac{1}{2 \Hb }\left(fR-2\Delta f\right)(w^3-w)-\frac{1}{2}(Q-Q^3)B .
\end{align*}
Since $B\geq \frac{fR-2\Delta f}{ \Hb }$, we have
\begin{align*}
0\geq w-w^3+Q^3-Q ,
\end{align*}
which implies  $w\geq Q$ as $Q\geq 1$.
Thus, $w\geq Q$ for all $t \ge t_0$. 

\medskip

\noindent Case 2:  $\min_{\Sigma_{t_0}} w< 1$. Define
\begin{align*}
\tilde{Q}(t)=\left(1+(C_2+\epsilon)\exp(-\int_{t_0}^t(A(s)-\epsilon)ds)\right)^{-\frac{1}{2}} .
\end{align*}
For $\epsilon$ small enough, we have
\begin{align*}
\tilde{Q}(t_0)=(1+C_2+\epsilon)^{-\frac{1}{2}}<\min_{\Sigma_{t_0}}w .
\end{align*}

Suppose now at some $t_1>t_0$, we have $\min_{\Sigma_{t_1} }  \left( w-\tilde{Q} \right) =0$ 
and, for $t_0\leq t\leq t_1$, we have $w-\tilde{Q}\geq 0$. Then at $t_1$, 
\begin{align*}
\frac{d}{dt}(w-\tilde{Q})\leq 0,\quad \nabla w=0,\quad \nabla^2w\geq 0 .
\end{align*}
Since
\begin{align*}
\frac{d}{dt}\tilde{Q}(t)=\frac{1}{2}(\tilde{Q}-\tilde{Q}^3)(A-\epsilon),
\end{align*}
we have 
\begin{align*}
0\geq -\frac{1}{2 \Hb }\left(fR-2\Delta f\right)(w^3-w)-\frac{1}{2}(\tilde{Q}-\tilde{Q}^3)(A-\epsilon).
\end{align*}
Since $A-\epsilon< \frac{fR-2\Delta f}{\Hb }$, the above implies 
\begin{align*}
0\geq \tilde{Q}-\tilde{Q}^3 .
\end{align*}
Contradict to the fact $\tilde{Q}<1$.
Since $\epsilon$ is arbitrary, we thus have 
\begin{align*}
w\geq \left(1+C_2\exp(-\int_{t_0}^tA(s)ds)\right)^{-\frac{1}{2}}.
\end{align*}

\medskip

Finally, note that $A(t)=\frac{n-1}{n}+O(e^{-\alpha t})$, we have
\begin{align*}
\exp(-\int_{t_0}^tA(t))\leq Ce^{-\frac{n-1}{n}t}\leq C\phi^{1-n} .
\end{align*}
Therefore, 
$ w\geq 1-C\phi^{1-n} $.
Similarly, 
$ w\leq 1+C\phi^{1-n} $.
Thus, we conclude 
\begin{align*}
|w-1|\leq C\phi^{1-n} .
\end{align*}
\end{proof}

\subsection{Asymptotic behavior of $w$}
Following  \cite{ST},  we consider the rescaled metric 
\begin{align*}
\tilde{g}_{ij}=\phi^{-2}g_{ij} 
\end{align*}
on each  $ \Sigma_t$. 
Here  we omit writing $t$  for the sake of  convenience. 
Note that by Lemma \ref{renormalized metric decay}, $\tilde{g}_{ij}$ converges to $\sigma_{ij}$ exponentially fast.

For any function $h$ and $l$, 
\begin{align*}
<\tilde{\nabla}h,\tilde{\nabla}l>_{\tilde{g}}=\phi^2<\nabla h,\nabla l>_g .
\end{align*}
Henceforth, for convenience, we  simply  write the above  as 
\begin{align*}
\tilde{\nabla}h\tilde{\nabla}l=\phi^2\nabla h\nabla l .
\end{align*}
Direct calculation gives 
\begin{align*}
\Delta=\phi^{-2}\tilde{\Delta}+(n-2)\phi^{-3} \tilde{\nabla}\phi\tilde{\nabla} .
\end{align*}
In terms of $\tilde g_{ij}$,   equation \eqref{equation w} becomes
\begin{align}\label{equation w-2}
\begin{split}
& \ 
-\frac{\partial w}{\partial t} +\frac{w^2}{\Hb \phi^2}\left( f\tilde{\Delta} w+(n-2)f\phi^{-1}\tilde{\nabla}\phi\tilde{\nabla}w+2\tilde{\nabla} w\tilde{\nabla} f\right) \\
= & \ \frac{1}{2 \Hb }\left(fR-2\phi^{-2}\tilde{\Delta} f-2(n-2)\phi^{-3}\tilde{\nabla}\phi\tilde{\nabla}f\right)(w^3-w) ,
\end{split} 
\end{align}
which can be re-written as 
\begin{align*}
\begin{split}
& \ -\frac{\partial w}{\partial t} +\tilde{\nabla}\left(\frac{fw^2}{ \Hb \phi^2}\tilde{\nabla} w\right)-\frac{2f}{H\phi^2}|\tilde{\nabla}w|^2 \\
= & \ w^2\tilde{\nabla}\left(\frac{f}{ \Hb \phi^2}\right)\tilde{\nabla}w-\frac{w^2}{ \Hb \phi^2}\left((n-2)f\phi^{-1}\tilde{\nabla}\phi
+2\tilde{\nabla} f\right)\tilde{\nabla}w\\
&+\frac{1}{2 \Hb }\left(fR-2\phi^{-2}\tilde{\Delta} f-2(n-2)\phi^{-3}\tilde{\nabla}\phi\tilde{\nabla}f\right)(w^3-w).
\end{split}
\end{align*}
By Lemma \ref{C^0 w}, this  is a uniformly parabolic PDE. In addition,  the term $-\frac{2f}{ \Hb \phi^2}|\tilde{\nabla}w|^2$ 
has a good sign and the coefficient of $\tilde{\nabla}w$ is uniformly bounded. Thus we may directly apply standard 
Moser iteration to conclude that $w\in C^\alpha$. 

By considering  the equation for $w-1$ and  applying Schauder estimate and Lemma \ref{C^0 w}, for any $k,l\geq 0$, we have
\begin{align}\label{high regu w}
|\left(\frac{\partial}{\partial t}\right)^k\tilde{\nabla}^l(w-1)|\leq C\phi^{1-n},
\end{align}
where $C$ depends only on $\Sigma_0,n$ and $k,l$.
As in \cite{ST}, we define
\begin{equation} \label{definition m}
\mathfrak{m}=\frac{1}{2}\phi^{n-1}(1-w^{-2}) .
\end{equation}

\begin{lemm} \label{estimate m}
There exists a constant $m_0$, such that
\begin{align*}
|\mathfrak{m}-m_0|+|\nabla_0\mathfrak{m}|+|\frac{\partial \mathfrak{m}}{\partial t}|\leq Ce^{-\alpha t} ,
\end{align*}
where $\nabla_0$ is the standard gradient on $\mathbb{S}^n$ and $C,\alpha$ depends only on $\Sigma_0$ and $n$.
\end{lemm}

\begin{proof}
By (\ref{high regu w}) and definition of $\mathfrak{m}$, for any $k,l\geq 0$, we have
\begin{align*}
|\left(\frac{\partial}{\partial t}\right)^k\tilde{\nabla}^l\mathfrak{m}|\leq C,
\end{align*}
where $C$ depends only on $\Sigma_0,n$ and $k,l$.
By (\ref{equation w-2}), $ \mathfrak{m}$ satisfies 
\begin{align*}
\frac{\partial \mathfrak{m}}{\partial t}=&\frac{n-1}{2}\phi^{n-2}\phi^\prime  (1-w^{-2})\frac{\partial r}{\partial t}+\phi^{n-1}w^{-3}\frac{\partial w}{\partial t}\\
=&\frac{n-1}{2}\phi^{n-2}\phi^\prime vf(1-w^{-2})+\frac{\phi^{n-3}}{ \Hb w}\left(f\tilde{\Delta} w+(n-2)f\phi^{-1}\tilde{\nabla}\phi\tilde{\nabla}w+2\tilde{\nabla} w\tilde{\nabla} f\right)\\
&-\frac{\phi^{n-1}}{2 \Hb w^3}\left(fR-2\phi^{-2}\tilde{\Delta} f-2(n-2)\phi^{-3}\tilde{\nabla}\phi\tilde{\nabla}f\right)(w^3-w) .
\end{align*}

Denote by $p$ any function that satisfies
\begin{align*}
|\left(\frac{\partial}{\partial t}\right)^k\tilde{\nabla}^lp|\leq Ce^{-\alpha t},
\end{align*}
for any $k,l\geq 0$, where $C,\alpha$ is uniform constants may depends on $k,l$. 
By Lemma \ref{C^1 estimate} and Lemma \ref{asymptotic second fundamental form}, we have
\begin{align*}
\phi^{-1}\tilde{\nabla}\phi=p,\quad \phi^{-1}\tilde{\nabla}f=p .
\end{align*}
Thus
\begin{align*}
\frac{\phi^{n-3}}{ \Hb w}f\phi^{-1}\tilde{\nabla}\phi\tilde{\nabla}w=\left(\phi^{n-1}\tilde{\nabla}w\right)\left(\phi^{-1}\tilde{\nabla}\phi\right)\frac{f}{\Hb w\phi^2}=p .
\end{align*}
Similarly,
\begin{align*}
\frac{\phi^{n-3}}{\Hb w}\tilde{\nabla} w\tilde{\nabla} f=\left(\phi^{n-1}\tilde{\nabla}w\right)\left(\phi^{-1}\tilde{\nabla}f\right)\frac{1}{\Hb w\phi}=p ,
\end{align*}
\begin{align*}
\frac{\phi^{n-1}}{\Hb w^3}\phi^{-2}\tilde{\Delta} f(w^3-w)=\left(\phi^{n-1}(1-w^{-2})\right)\left(\phi^{-1}\tilde{\Delta}f\right)\frac{1}{\Hb \phi}=p ,
\end{align*}
\begin{align*}
\frac{\phi^{n-1}}{\Hb w^3}\phi^{-3}\tilde{\nabla}\phi\tilde{\nabla}f(w^3-w)=\left(\phi^{n-1}(1-w^{-2})\right)\left(\phi^{-1}\tilde{\nabla}f\right)\left(\phi^{-1}\tilde{\nabla}\phi\right)\frac{1}{\Hb \phi}=p .
\end{align*}
Hence, 
\begin{align*}
\frac{\partial \mathfrak{m}}{\partial t}&=\frac{n-1}{2}\phi^{n-2}\phi^\prime vf(1-w^{-2})+\frac{\phi^{n-3}f}{\Hb w}\tilde{\Delta} w-\frac{\phi^{n-1}fR}{2\Hb}(1-w^{-2})+p .
\end{align*}
On the other hand, 
\begin{align*}
\begin{split}
& \ \frac{n-1}{2}\phi^{n-2}\phi^\prime vf(1-w^{-2})-\frac{\phi^{n-1}fR}{2 \Hb }(1-w^{-2}) \\
= & \ \frac{\phi^{n-1}(1-w^{-2})}{2}f\left((n-1)\phi^{-1}\phi^\prime v-\frac{R}{\Hb }\right)  = p, 
\end{split} 
\end{align*}
Therefore, 
\begin{align*}
\frac{\partial \mathfrak{m}}{\partial t}=\frac{\phi^{n-3}f}{ \Hb w}\tilde{\Delta} w+p .
\end{align*}

Note that
\begin{align*}
\tilde{\nabla}\mathfrak{m}=\frac{n-1}{2}\phi^{n-2}\tilde{\nabla}\phi(1-w^{-2})+\phi^{n-1}w^{-3}\tilde{\nabla}w ,
\end{align*}
and
\begin{align*}
\tilde{\Delta}\mathfrak{m}=&\frac{(n-1)(n-2)}{2}\phi^{n-3}|\tilde{\nabla}\phi|^2(1-w^{-2})+\frac{n-1}{2}\phi^{n-2}\tilde{\Delta}\phi(1-w^{-2})\\
&+2(n-1)\phi^{n-2}w^{-3}\tilde{\nabla}\phi\tilde{\nabla}w-3\phi^{n-1}w^{-4}|\tilde{\nabla}w|^2+\phi^{n-1}w^{-3}\tilde{\Delta}w .
\end{align*}
Thus, 
\begin{align*}
\tilde{\Delta}\mathfrak{m}=-3\phi^{n-1}w^{-4}|\tilde{\nabla}w|^2+\phi^{n-1}w^{-3}\tilde{\Delta}w+p=\phi^{n-1}w^{-3}\tilde{\Delta}w+p .
\end{align*}
Therefore, 
\begin{align*}
\frac{\partial \mathfrak{m}}{\partial t}&=\frac{\phi^{n-3}f}{\Hb w}\left(\phi^{1-n}w^3\tilde{\Delta}\mathfrak{m}+\phi^{1-n}w^3p\right)+p\\
&=\frac{fw^2}{ \Hb \phi^2}\tilde{\Delta} \mathfrak{m}+p=\frac{1}{n^2}\tilde{\Delta} \mathfrak{m}+p .
\end{align*}

By Lemma \ref{renormalized metric decay}, we have $\tilde{g}_{ij}=\sigma_{ij}+p$, where $\sigma_{ij}$ is the standard metric on $\mathbb{S}^n$. Thus $\tilde{\Delta} \mathfrak{m}=\Delta_0\mathfrak{m}+p$, where $\Delta_0$ is the standard Laplacian on $\mathbb{S}^n$.
Now, by Lemma 2.6 in \cite{ST}, we conclude that there exists a constant $m_0$, such that
\begin{align*}
|\mathfrak{m}-m_0|+|\nabla_0\mathfrak{m}|+|\frac{\partial \mathfrak{m}}{\partial t}|\leq Ce^{-\alpha t}.
\end{align*}
\end{proof}
 
Lemma \ref{estimate m} directly implies the following asymptotic expansion of $ w$.

\begin{lemm} \label{cor-1-w-2}
As $ t \to \infty$,  $w$ satisfies
\begin{align*}
w=1+m_0\phi^{1-n}+p ,
\end{align*}
where $p=O(\phi^{1-n-\alpha})$ and $|\nabla_0p|=O(\phi^{-n-\alpha})$. Here  $\nabla_0$ denotes  the standard gradient on $(\mathbb{S}^n, \sigma)$.
\end{lemm}

\subsection{ADM mass of $g_\eta$}
We now  verify that the metric $g_\eta$ is asymptotically flat and 
we compute its ADM mass.  Note that
\begin{align*}
g_\eta=f^2dt^2+g_t+(\eta^2-f^2)dt^2=\bar{g}+(w^2-1)f^2dt^2,
\end{align*}
where $\bar{g}$ is the metric on the  Schwarzschild manifold  $ \SM$ with mass  $m$.

Let $r$ be the radial coordinate in \eqref{equation SM}. 
Let $ z = (z_1, \ldots, z_{n+1})$ denote the standard rectangular coordinates on the background Euclidean space
$$ \mathbb{R}^{n+1} = \left( [ 0 , \infty) \times \mathbb{S}^n , d r^2 + r^2 \sigma \right) . $$
Writing  $ \bar g = \bg_{ij} d z_i d z_j $ and $g_\eta=g_{ij}dz_idz_j$,  we  have 
\begin{align} \label{g-bg}
g_{ij}=\bar{g}_{ij}+b_{ij},
\end{align}
where 
\begin{align*}
b_{ij}=(w^2-1)f^2\frac{\partial t}{\partial z_i}\frac{\partial t}{\partial z_j}.
\end{align*}
We need to analyze the term $\frac{\partial t}{\partial z_i}$. As $r=|z|$, 
\begin{equation*}
\partial_{z_i} =  \frac{ z_i}{r}  \partial_r + ( \partial_{z_i} )^T ,
\end{equation*}
where $ ( \partial_{z_\alpha} )^T  $ is tangential to $\mathbb{S}^n$.
By definition, 
\begin{equation*}
\frac{\partial t}{\partial z_i}=d t(\partial_{z_i})  = \langle \bar \nabla t, \partial_{z_i} \rangle_{\bar g} = f^{-1} \langle \nu , \partial_{z_i} \rangle_{\bar g} .
\end{equation*}
Plugging in  $ \nu = \frac{1}{v} (\partial_r - \frac{r^j\partial_j}{\phi^2}) $, we have 
\begin{equation}\label{partial t partial i}
\begin{split}
\frac{\partial t}{\partial z_i}
= & \ \frac{1}{fv}  \left(\frac{ z_i}{r}    - \frac{r^j}{\phi^2} \langle  \partial_j  ,  ( \partial_{z_i} )^T \rangle_{\bar g}  \right) \\
= & \ \frac{1}{fv}  \left(\frac{ z_i}{r}   - r^j\langle  \partial_j  ,  ( \partial_{z_i} )^T \rangle_{\sigma}  \right) .
\end{split}
\end{equation}
Thus, 
\begin{equation*} 
b_{ij}= \frac{w^2-1}{v^2}  \left(  \frac{ z_i}{r}   - r^k\langle  \partial_k  ,  ( \partial_{z_i} )^T \rangle_{\sigma} \right)\left(  \frac{ z_j}{r}   - r^l\langle  \partial_l  ,  ( \partial_{z_j} )^T \rangle_{\sigma} \right) . 
\end{equation*}
By  Lemma \ref{cor-1-w-2}, Lemma \ref{C^1 estimate} and the fact that $|( \partial_{z_i} )^T|\leq \frac{1}{r}$, we have 
\begin{equation*}
| b_{ij} | = O ( | z |^{1-n} ). 
\end{equation*}
Similarly computation gives 
\begin{equation*}
| z | | \partial_z b_{ij} | + | z |^2 | \partial_z \partial_z  b_{ij} |  = O ( | z |^{1-n} ).
\end{equation*}
This  shows  that $ g_\eta$ is asymptotically flat. 

\begin{lemm}\label{ADM mass}
The ADM mass of $g_\eta=\eta^2dt^2+g_t$ equals $m + m_0$.
\end{lemm}

\begin{proof}
The ADM mass of $g_\eta$ is given  by
\begin{align*}
\frac{1}{2 n \omega_n} 
\lim_{r\rightarrow \infty}\int_{\mathbb{S}^n}\left(\frac{\partial {g_\eta}_{ij}}{\partial z_i}-\frac{\partial {g_\eta}_{ii}}{\partial z_j}\right)r^{n-1}z_j d\sigma .
\end{align*}
By \eqref{g-bg} and the fact that the ADM mass of $\bg $ is $m$, the above limit is equal to
\begin{align} \label{mass geta}
m + \frac{1}{2 n \omega_n} \lim_{r\rightarrow \infty}\int_{\mathbb{S}^n}\left(\frac{\partial b_{ij}}{\partial z_i}-\frac{\partial b_{ii}}{\partial z_j}\right)r^{n-1}z_j d\sigma.
\end{align}
Thus it suffices  to calculate $\frac{\partial b_{ij}}{\partial z_i}$ and $\frac{\partial b_{ii}}{\partial z_j}$. We have
\begin{align*}
\frac{\partial b_{ij}}{\partial z_i} =2wf^2\frac{\partial w}{\partial z_i}\frac{\partial t}{\partial z_i}\frac{\partial t}{\partial z_j}
+2(w^2-1)f\frac{\partial f}{\partial z_i}\frac{\partial t}{\partial z_i}\frac{\partial t}{\partial z_j}
+(w^2-1)f^2\left(\frac{\partial^2 t}{\partial z_i^2}\frac{\partial t}{\partial z_j}+\frac{\partial t}{\partial z_i}\frac{\partial^2 t}{\partial z_j\partial z_i}\right).
\end{align*}
Similarly,
\begin{align*}
\frac{\partial b_{ii}}{\partial z_j}=2wf^2\frac{\partial w}{\partial z_j}\left(\frac{\partial t}{\partial z_i}\right)^2+2(w^2-1)f\frac{\partial f}{\partial z_j}\left(\frac{\partial t}{\partial z_i}\right)^2+2(w^2-1)f^2\frac{\partial t}{\partial z_i}\frac{\partial^2 t}{\partial z_j\partial z_i} .
\end{align*}
Thus,
\begin{align*}
\frac{\partial b_{ij}}{\partial z_i}-\frac{\partial b_{ii}}{\partial z_j}=
&2wf^2\frac{\partial t}{\partial z_i}\left(\frac{\partial w}{\partial z_i}\frac{\partial t}{\partial z_j}-\frac{\partial w}{\partial z_j}\frac{\partial t}{\partial z_i}\right)
+2(w^2-1)f\frac{\partial t}{\partial z_i}\left(\frac{\partial f}{\partial z_i}\frac{\partial t}{\partial z_j}-\frac{\partial f}{\partial z_j}\frac{\partial t}{\partial z_i}\right)\\
&+(w^2-1)f^2\left(\frac{\partial^2 t}{\partial z_i^2}\frac{\partial t}{\partial z_j}-\frac{\partial t}{\partial z_i}\frac{\partial^2 t}{\partial z_j\partial z_i}\right)
\end{align*}
By Lemma \ref{cor-1-w-2}, we have
\begin{align*}
\frac{\partial w}{\partial z_i}=(1-n)m_0\phi^{-n-1}z_i+O(\phi^{-n-\alpha}).
\end{align*}
By Lemma \ref{C^1 estimate} and (\ref{partial t partial i}), we have
\begin{align*}
\frac{\partial t}{\partial z_i}=\frac{1}{f}\frac{z_i}{\phi}+O(\phi^{-1-\alpha}).
\end{align*}
Therefore, 
\begin{align*}
2wf^2\frac{\partial t}{\partial z_i}\left(\frac{\partial w}{\partial z_i}\frac{\partial t}{\partial z_j}-\frac{\partial w}{\partial z_j}\frac{\partial t}{\partial z_i}\right)=O(\phi^{-n-\alpha}).
\end{align*}
On other other hand, by Lemma \ref{asymptotic second fundamental form} and straightforward computation,
\begin{align*}
\frac{\partial f}{\partial z_i}=\frac{z_i}{n\phi}+O(\phi^{-\alpha}).
\end{align*}
Thus,
\begin{align*}
2(w^2-1)f\frac{\partial t}{\partial z_i}\left(\frac{\partial f}{\partial z_i}\frac{\partial t}{\partial z_j}-\frac{\partial f}{\partial z_j}\frac{\partial t}{\partial z_i}\right)=O(\phi^{-n-\alpha}).
\end{align*}
Again by Lemma \ref{asymptotic second fundamental form}, Lemma \ref{C^1 estimate} and (\ref{partial t partial i}) , we have
\begin{align*}
\frac{\partial^2 t}{\partial z_i^2}=\frac{n(n-2)}{\phi^2}+O(\phi^{-2-\alpha}), \ 
\frac{\partial^2 t}{\partial z_i\partial z_j}=-\frac{2n}{\phi^4}z_iz_j+O(\phi^{-2-\alpha}).
\end{align*}
Thus,
\begin{align*}
(w^2-1)f^2\left(\frac{\partial^2 t}{\partial z_i^2}\frac{\partial t}{\partial z_j}-\frac{\partial t}{\partial z_i}\frac{\partial^2 t}{\partial z_j\partial z_i}\right) 
=& \ 2m_0\phi^{1-n}f^2\frac{n^2}{\phi^2}\frac{z_j}{f\phi} \\ 
= & \ 2nm_0\phi^{-1-n}z_j+O(\phi^{-n-\alpha}).
\end{align*}
Therefore, we conclude 
\begin{align*}
\frac{\partial b_{ij}}{\partial z_i}-\frac{\partial b_{ii}}{\partial z_j}=2nm_0\phi^{-1-n}z_j+O(\phi^{-n-\alpha}), 
\end{align*}
which implies that the ADM mass of $g_\eta$ is $ m + m_0 $ by \eqref{mass geta}. 
\end{proof}

\begin{rema}
A more geometric way to compute  the  ADM mass of $g_\eta$ is as follows.  
The  foliation $\{ \Sigma_t \}$  is a family of  nearly  round hypersurfaces  according to  Definition 2.1 in \cite{MTX}.
Thus, if $\mathfrak{m} (g_\eta) $ is  the  mass of $g_\eta$, then by Theorem 1.2  in \cite{MTX}, 
\begin{equation}
\begin{split}
\mathfrak{m} (g_\eta)  =  & \ \lim_{ t \rightarrow \infty} \frac{1}{2 n (n-1) \omega_n} \left( \frac{ | \Sigma_t | }{ \omega_n} \right)^\frac{1}{n} 
\int_{\Sigma_t} \left( R - \frac{n-1}{n} H^2_\eta \right) d \sigma \\
 =  & \ \lim_{ t \rightarrow \infty} \frac{1}{2 n (n-1) \omega_n} \left( \frac{ | \Sigma_t | }{ \omega_n} \right)^\frac{1}{n} 
\int_{\Sigma_t} \left( R - \frac{n-1}{n} \Hb^2 + \frac{n-1}{n} ( 1 - w^{-2} ) \Hb^2 \right) d \sigma  \\
=  & \ m + m_0 ,
\end{split}
\end{equation}
where we have  used   the fact $\bar g$ has mass $m$ and 
\begin{equation} \label{eq-1-w-negative-2}
 \lim_{ t \rightarrow \infty} \int_{\Sigma_t}   ( 1 - w^{-2} ) \Hb^2  d \sigma  = 2 n^2 \omega_n  m_0  ,
\end{equation}
which follows from Lemma \ref{cor-1-w-2}, Lemma  \ref{asymptotic second fundamental form} and Lemma
 \ref{renormalized metric decay}.
\end{rema}

\begin{lemm}\label{asymptotic integral}
\begin{align*}
\lim_{t\rightarrow \infty}\int_{\Sigma_t}N( \Hb -H_\eta) d \sigma =\lim_{t\rightarrow \infty}\int_{\Sigma_t} N \Hb (1-w^{-1})d \sigma =nm_0\omega_n
\end{align*}
\end{lemm}

\begin{proof}
Similar to \eqref{eq-1-w-negative-2},  this is a direct consequence of 
 Lemma \ref{cor-1-w-2}, Lemma  \ref{asymptotic second fundamental form} and Lemma \ref{renormalized metric decay}.
\end{proof}

We summarize the results in Lemmas \ref{C^0 perturbed}, \ref{cor-1-w-2},   \ref{ADM mass} and  \ref{asymptotic integral}
in the following theorem.

\begin{theo} \label{BST extension} 
Let  $\Sigma^n \subset \SM $ be a closed, star-shaped, $2$-convex hypersurface with  
$ \bRic (\nu, \nu) \le 0 $, where  $ \bRic(\cdot, \cdot) $ is the Ricci curvature of  $\SM$ and $\nu$ is the outward unit normal to $ \Sigma$. 
Let $ \mathbb{E}$ denote the exterior of $ \Sigma^n$ in $ \SM$,  which is swept by a family of star-shaped hypersurfaces $ \{ \Sigma_t \}_{ 0 \le t \le \infty }$ 
that is   a smooth solution  to 
\begin{align*} 
\frac{\p X}{\p t} {}=\frac{n-1}{2n}\frac{\sigma_1}{\sigma_2} \nu 
\end{align*}
with initial condition $ \Sigma_0 = \Sigma^n $.  
On $ \E$, writing  the Schwarzschild metric $\bg $ as 
\begin{align*}
\bar{g}=f^2dt^2+g_t ,
\end{align*}
where $g_t$ is the induced metric on $\Sigma_t $ and 
$f =\frac{n-1}{2n}\frac{\sigma_1}{\sigma_2 }  $.  
Then, given any smooth  function $ \psi > 0 $ on $ \Sigma$, 
there exists a smooth function $ \eta > 0 $ on $ \E$ such that 
\begin{enumerate}
\item[(i)]  $ \eta |_{\Sigma} = \psi$, 
the metric $g_\eta = \eta^2 dt^2 + g_t $ has zero scalar curvature, and 
$ \eta $ satisfies 
$$ f^{-1} \eta =1+m_0\phi^{1-n}+p , $$
where $m_0$ is a constant, $p=O(\phi^{1-n-\alpha})$ and $|\nabla_0p|=O(\phi^{-n-\alpha})$; 

\item[(ii)]  the Riemannian manifold $(\E,   g_\eta)$ is asymptotically flat;  and 

\item[(iii)] the ADM mass $\m (g_\eta)$ of $g_\eta$ is given by 
\begin{align*}
 \m(g_\eta)  =   m + m_0 
 =  m + \lim_{t\rightarrow \infty} \frac{1}{ n \omega_n} \int_{\Sigma_t}N(\Hb -H_\eta)d \sigma .
 \end{align*}
\end{enumerate}
\end{theo}

\begin{rema}
Since $(\E, g_\eta) $   is foliated by $ \{ \Sigma_t \}_{ 0 \le t \le \infty }$,
which has positive mean curvature for each $t$, 
the boundary  $ \p \E = \Sigma $ is  outer minimizing in $(\E, g_\eta)$, meaning that  
 $\Sigma$  minimizes area among all hypersurfaces in $\E$ that 
enclose $ \Sigma $.
\end{rema}

\section{Geometric Applications} \label{sec apps}
In this section, we give  applications of results  in Sections \ref{sec-mono-static} -- \ref{sec-warped-metric}. 
First, we prove  Theorem \ref{thm-intro-main}.

\begin{proof}[Proof of Theorem \ref{thm-intro-main}]
Let $(\Omega^{n+1},\fg)$ be a compact manifold given   in Theorem \ref{thm-intro-main}.
By  assumptions   (i), (ii) and  the standard geometric measure theory,  $\Sh$  minimizes area among 
all closed hypersurfaces in $(\Omega, \fg)$ that encloses $\Sh$.

Let $ \E $ denote the exterior of $ \Sigma^n $ in $ \SM$. 
Let $ \eta > 0 $ be the smooth function  on $ \E $ given by Theorem \ref{BST extension}
with an initial condition 
\be \label{H-fH}
 \eta |_{\Sigma} = f |_\Sigma  \fH^{-1} H_m .
\ee
This condition  implies  that the mean curvature of $ \Sigma$ in $(\E, g_\eta)$ agrees 
with the mean curvature $\fH$ of $ \So$ in $(\Omega, \fg)$. 
Since $ \So$ is isometric to $ \Sigma = \p \E$, we can attach   $(\E, g_\eta)$  to  $(\Omega, \fg) $ along $ \Sigma = \So $ 
by  matching the  Gaussian neighborhood of $\Sigma $ in $(\E, g_\eta)$ to 
that of $\So$ in   $(\Omega, \fg )$.
Denote the resulting manifold by  $\hat M$ and its metric by $\hat h$. 
By construction,  $\hat h$ is Lipschitz across $\Sigma$ and smooth everywhere else on $\hat M$;
$ \hat h$  has nonnegative scalar curvature away from $\Sigma$; and the mean curvature of $ \Sigma$
from both sides in $(\hat M, \hat h) $ agree.
Moreover,  $ \partial {\hat M} = \Sh $ is a minimal hypersurface that is 
outer minimizing in $ (\hat M, \hat h)$. This  outer minimizing property of $ \Sh$  is guaranteed by the 
fact that  $ \Sigma$ is outer minimizing in $(\E, g_\eta) $ and 
$\Sh$ minimizes area among  closed hypersurfaces in $(\Omega, \fg)$ that encloses $\Sh$.
On such an $(\hat M, \hat h)$, it is known that the Riemannian Penrose inequality, i.e. Theorem \ref{thm-intro-RPI},  still holds.
(For a proof of this claim, see page 279-280 in \cite{Miao09} for the case $n=2$ and  Proposition  3.1 in \cite{MM} for $ 2 \le n \le 6$).  
Therefore, we have
\be \label{eq-app-RPI}
\m (g_\eta) \ge \frac12 \left( \frac{|\Sh |}{\omega_{n}}  \right)^{\frac{n-1}{n}}. 
\ee 
By  (iii) in Theorem \ref{BST extension}, this gives
\be \label{pf 1}
m + \lim_{t\rightarrow \infty} \frac{1}{ n \omega_n} \int_{\Sigma_t}N(\Hb -H_\eta) 
d \sigma   \ge \frac12 \left( \frac{|\Sh |}{\omega_{n}}  \right)^{\frac{n-1}{n}}. 
\ee

On the other hand,
since $ \frac{\partial N}{\p \nu} > 0 $ as  $ \Sigma_t $ is star-shaped and
$ \Sigma_t $ has  positive $\sigma_1$ and  $ \sigma_2$ in $\SM$, 
Corollary \ref{cor-1} applies with $ (\N, \bg)$ given by $ \SM$ to show  that 
$$  \int_{\Sigma_t}N( \Hb -H_\eta)d \sigma  $$ is monotone nonincreasing. 
At $\Sigma = \Sigma_0 $, we have $ H_m = \Hb$ and $ H = H_\eta$. 
Therefore, 
\begin{equation} \label{pf 2}
\begin{split} 
  \int_{\Sigma }N(H_m - \fH )d \sigma  = & \  \int_{\Sigma_0}N( \Hb - H_\eta)d \sigma  \\
   \ge & \  \lim_{t\rightarrow \infty}  \int_{\Sigma_t}N( \Hb -H_\eta)d \sigma .
\end{split}
\end{equation}
Therefore, it follows from \eqref{pf 1} and \eqref{pf 2} that
\be \label{pf main}
m + \frac{1}{ n \omega_n}  \int_{\Sigma }N( H_m - \fH )d \sigma  \ge \frac12 \left( \frac{|\Sh |}{\omega_{n}}  \right)^{\frac{n-1}{n}},
\ee
which proves \eqref{eq-intro-main}. If equality in \eqref{pf main} holds, 
then $  \int_{\Sigma_t}N( \Hb -H_\eta)d \sigma  $ is a constant for all $t$. 
By Corollary \ref{cor-1}, we have  $ \eta = f $ on $ \E$,  hence $ \fH = H_m $ by \eqref{H-fH}.
Consequently, 
$$ m =  \frac12 \left( \frac{|\Sh |}{\omega_{n}}  \right)^{\frac{n-1}{n}} .$$
This completes the proof of Theorem \ref{thm-intro-main}. 
\end{proof}

\begin{rema}
We conjecture that, when equality in \eqref{eq-intro-main} holds, 
$(\Omega, \fg)$ is isometric to the domain enclosed  by  $\Sigma$ and the horizon boundary $\Sh^S$ in $\SM$. 
It is clear from the  above proof that in this case  \eqref{eq-app-RPI} becomes equality. 
Thus, if one can establish the rigidity statement for  the Riemannian Penrose inequality  
on  manifolds with corners (cf. \cite{Miao02, ST, MS}), then this conjecture will follow.
\end{rema}

Next, we note an implication of Corollary \ref{cor-1} and Theorem \ref{BST extension} on
the concept of Bartnik mass \cite{Bartnik}. Given a pair $(g, H)$, 
where  $g$ is a metric  and $H$ is a function on $\mathbb{S}^2$, 
the Bartnik mass of $(g, H)$, which we denote by $ \mB(g, H) $, can be defined by 
\begin{align*}
\mB  (g,H)  = \inf \left\{  \m (h ) \, \vert \, (M, h ) \text{\ is  an  admissible extension of }(\mathbb{S}^2 , g, H) \right\}.
\end{align*}
Here $\m (h)$ is the ADM mass of $ (M, h)$ which is an asymptotically flat  $3$-manifold with boundary $\p M$.   
$(M, h)$  is called an \emph{admissible extension 
of $(\mathbb{S}^2 ,g,H)$} provided  $(M, h)$ has nonnegative scalar curvature,  $ \p M$  is isometric to 
$(\mathbb{S}^2,g)$, and the mean curvature of $ \p M$ in $(M, h)$ equals   $H$ under the identification 
of $ \p M$ with $ (\mathbb{S}^2, g)$ via the isometry.  Moreover, it is assumed that  either $(M, h)$ 
 contains no closed minimal surfaces or  $\p M$ is outer minimizing in $(M, h)$ (see \cite{Bartnik, B, B04,HI01}). 

\begin{theo} \label{thm Bmass}
Given a pair $(g, H)$  on  $ \mathbb{S}^2$, suppose $H > 0 $ and $(\mathbb{S}^2, g)$ is isometric to 
a closed, star-shaped, convex surface $\Sigma$ with  $ \bRic (\nu, \nu) \le 0 $ in a spatial 
Schwarzschild manifold $ \SM$ with $m>0$. Then
$$ \mB (g, H) \le  m + \frac{1}{ 8 \pi } \int_{\Sigma}N( H_m-H )  \, d \sigma . $$
\end{theo}

\begin{proof}
Taking $ n = 2 $ in Theorem \ref{BST extension}, let $ \eta $ be the function given on $ \E$ with an initial condition 
$  \eta |_{\Sigma} = f |_\Sigma  H^{-1} H_m $.  The asymptotically flat manifold $(\E , g_\eta)$ is an admissible extension of $(\mathbb{S}^2, g, H)$.
Therefore, by (iii) in Theorem \ref{BST extension}, 
\be \label{Bmass 1}
\begin{split}
 \mB (g, H)  \le & \ \m (g_\eta) 
 =   m + \lim_{t\rightarrow \infty} \frac{1}{ 8 \pi } \int_{\Sigma_t}N( \Hb -H_\eta) \, d \sigma .
\end{split}
\ee
By Corollary \ref{cor-1}, 
\begin{equation} \label{Bmass 2}
\begin{split} 
  \int_{\Sigma }N(H_m - H ) \, d \sigma  
   \ge  \lim_{t\rightarrow \infty}  \int_{\Sigma_t}N( \Hb -H_\eta) \, d \sigma .
\end{split}
\end{equation}
Theorem \ref{thm Bmass} follows from \eqref{Bmass 1} and \eqref{Bmass 2}.
\end{proof} 

\begin{rema}
Indeed our method shows  
the following  is true -- 
given a pair $(g, H)$  on  $ \mathbb{S}^2$, 
suppose $(\mathbb{S}^2, g)$ is isometric to the boundary of a static,   asymptotically flat manifold 
$(\N^3, \bg)$ with a positive static potential $N$. Suppose  $(\N, \bg)$ satisfies:
\begin{enumerate} 
\item[(i)] $ \Sigma = \p \N$ has positive $ \sigma_1$ and $ \sigma_2$; 
\item[(ii)] the inverse curvature flow 
\eqref{eq-intro-flow-sigma} in $(\N, \bg)$, with initial condition $ \Sigma_0 = \Sigma$, 
admits a long time, smooth solution $ \{ \Sigma_t \}_{0 \le t < \infty}$ with $ \frac{\p N}{\p \nu} > 0 $;  and 
\item[(iii)] the warped metric $g_\eta$ defined by  \eqref{perturbed metric}, satisfying  $ R(g_\eta) = 0 $ on $ \N$
and $ H_\eta = H$ at $ \Sigma$, can be constructed on $\N$  such that  $ g_\eta$ 
is asymptotically flat with 
$$  \m(g_\eta) 
 =  \m(\bg)  + \lim_{t\rightarrow \infty} \frac{1}{ 8 \pi } \int_{\Sigma_t}N(\Hb -H_\eta) \, d \sigma .$$
\end{enumerate}
Then, by Corollary \ref{cor-1}, the Bartnik mass of $(g, H)$ satisfies 
\be \label{eq-Bmass-bound-g}
 \mB (g, H) \le  \m (\bg) + \frac{1}{ 8 \pi } \int_{\Sigma}N( \Hb - H ) \, d \sigma .
\ee
Here  $ \m (\bg)$ is the ADM mass of $(\N, \bg)$ and $\Hb$ is the mean curvature of $ \Sigma $ in $(\N, \bg)$.
Estimate \eqref{eq-Bmass-bound-g} appeared as  Conjecture 4.1 in \cite{Miao-ICCM-07}.
\end{rema}

\section{Limits along isomeric embeddings of large spheres into Schwarzschild manifolds}
\label{sec-iso-emb}

In this section, we prove Theorem \ref{thm-intro-iso-emb} which was inspired by the results of Fan, Shi and Tam \cite{FST}.
We divide the proof into two parts, the existence of the embedding and the calculation of the limits. 

\subsection{Isometric Embedding of large spheres}
In \cite{Nirenberg}, Nirenberg shows that a $2$-sphere with positive Gauss curvature can be isometrically embedded in $ \mathbb{R}^3$ 
as a strictly convex surface. 
By adopting  the iteration scheme used in the proof of the openness part in \cite{Nirenberg}, 
one can verify  that a perturbation of a standard round sphere can be isometrically embedded in a $3$-dimensional Schwarzschild 
manifold with small mass. This assertion, which is the main tool we use in this section, 
is  indeed a special case of  \cite[Theorem 1]{LW} (see also \cite{CPM}).

\begin{prop}[\cite{CPM, LW}]  \label{prop-iso}
Let $ \sigma$ be the standard  metric on $ \mathbb{S}^2$ with area $4 \pi $. There exists $\epsilon > 0 $ and $ \delta > 0 $, 
such that if $ \tilde \sigma$  is a metric on $ \mathbb{S}^2$ with $ \| \tilde \sigma - \sigma \|_{C^{2,\alpha}} < \epsilon $
and  if $m$ is a  positive constant with  $  m < \delta $, then 
there exists an isometric embedding $\tilde X$ of $(\mathbb{S}^2, \tilde \sigma)$ in 
\be \label{eq-SM-rot}
 \mathbb{M}^3_m = \left( [2m, \infty) \times \mathbb{S}^2,  \frac{1}{1 - \frac{2m}{\rho} } d \rho^2 + \rho^2 \sigma \right). 
 \ee
 Moreover, $ \tilde X$ can be chosen so that 
\be \label{est-embedding}
 \| \tilde X - X \|_{C^{2,\alpha} } \le C \| \tilde \sigma - \sigma \|_{C^{2, \alpha} } .
 \ee
Here $ X$ is the isometric embedding of $(\mathbb{S}^2, \sigma)$  in $ \mathbb{M}^3_m$
given by $ X (\omega ) = (1, \omega)$, $ \forall \ \omega \in \mathbb{S}^2$. 
\end{prop}

\begin{rema}
Estimate \eqref{est-embedding} is not stated in the statement of  theorems in \cite{CPM, LW},
but it follows from both proofs therein. 
\end{rema}

We now consider an asymptotically flat $3$-manifold $(M, \fg)$  given in Theorem \ref{thm-intro-iso-emb}. 
Precisely, this means that, outside a compact set, $ M$ is diffeomorphic to $ \mathbb{R}^3$ 
minus a ball and with respect to the standard coordinates on $\mathbb{R}^3$, 
$\fg$  satisfies  $  \fg_{ij} = \delta_{ij} + p_{ij} $
 with
 \begin{equation} \label{eq-AF3}
  | p_{ij} | + | x | | \partial p_{ij} | + | x |^2 | \partial \partial  p_{ij} | + | x |^3 | \partial \partial \partial  p_{ij} | = O ( | x |^{-\tau} )
\end{equation}
for some constant $\tau > \frac12 $, 
where $ \partial $ denotes the partial derivative.  
Moreover, it is assumed that the scalar curvature  of $ \fg $ is integrable on $(M, \fg)$. Under such assumptions, 
the ADM mass of $(M, \fg)$ is well defined, which  we will denote  by $ \m $. 

Given a large constant $ r > 0 $,  let $ S_r = \{ | x | = r \}$ denote the  coordinate sphere in 
$(M, \fg)$.  Let $g_r$ be  the induced metric on $S_r$ and let $ \tilde g_r = r^{-2} g_r$.
Identifying $ S_r$ with  $ \mathbb{S}^2 = \{ \ | y | = 1 \}$ via a map $ y = r^{-1} x $, 
one  can deduce from  \eqref{eq-AF3} that 
\begin{align} \label{eq-tgr}
\|\tilde{g}_r-\sigma\|_{C^3}\leq Cr^{-\tau}
\end{align}
(see (2.17) in \cite{FST} for instance). 
Here  
$ \sigma$ is the standard  metric on $ \mathbb{S}^2 $ with area $4 \pi$ 
and  $ C> 0 $ is a constant  independent on $ r$. 

Let $ m > 0$ be any fixed constant. Define $ m_r = r^{-1} m $. 
Applying  Proposition \ref{prop-iso} and \eqref{eq-tgr}, we conclude, for sufficiently large $r$, 
there exists an isometric embedding 
$$ \tXr : (\mathbb{S}^2, \tilde g_r) \longrightarrow \mathbb{M}^3_{m_r} 
= \left( [2 m_r, \infty) \times \mathbb{S}^2,  \frac{1}{1 - \frac{2m_r}{\rho } } d \rho^2 + \rho^2 \sigma \right)  $$
satisfying 
\be \label{eq-txr}
\| \tXr - \Xz \|_{C^{2, \alpha}}   = O (r^{-\tau} ) ,
\ee
where  $ \Xz (\omega ) = (1, \omega) $, $ \forall \ \omega \in \mathbb{S}^2 $. 
It follows from  \eqref{eq-txr} that  $ \tXr (\mathbb{S}^2)$ is star-shaped and convex; moreover, if 
$ \tilde \nu_r $ is  the outward unit normal to $ \tXr (\mathbb{S}^2 )$,
then 
\be \label{eq-nur}
\tilde \nu_r =  \left( 1 +  O ( r^{- \tau} ) \right) \partial_\rho + O ( r^{- \tau} ) \partial_{\omega_1} + O ( r^{- \tau} ) \partial_{\omega_2} ,
\ee
where  $\omega_i$, $ i = 1, 2$, are local coordinates on $ \mathbb{S}^2$.  

Let $ \bRic_r $ denote  the  Ricci curvature of $\mathbb{M}^3_{m_r}$. In the rotationally symmetric 
form,  it  is given by 
$$
\bRic_r   =   m_r \rho^{-3}  \Psi ,
$$
where
$$
\Psi =  - \frac{2}{1 - \frac{2 m_r} {\rho} } d \rho^2 + \rho^2 \sigma  .
$$
By  \eqref{eq-txr}  and \eqref{eq-nur},
\be \label{eq-Ricr-nur}
\begin{split}
\bRic _r(\tilde \nu_r,  \tilde \nu_r) = & \ m_r \rho^{-3} \Psi (\tilde \nu_r,  \tilde \nu_r) \\
= & \ - 2 m_r  \left( 1 +  O ( r^{- \tau} ) \right)   .
\end{split}
\ee
In particular,  $ \bRic_r ( \tilde \nu_r,  \tilde \nu_r) < 0 $ for  large $r$.

The  map $ \tXr$ leads   to an isometric embedding of $(S_r, g_r)$ in $\mathbb{M}^3_m$
because   $ \mathbb{M}^3_{m}$ and $\mathbb{M}^3_{m_r}$ only differer by a constant scaling.
More precisely,  consider the map 
$$ F_r : \mathbb{M}^3_{m_r} \longrightarrow \mathbb{M}^3_{m}$$
where
$ F_r ( \rho, \omega) = ( r \rho , \omega) $. 
Define  $ X_r = F_r \circ \tXr $, 
then 
$$X_r: (S_r, g_r) \longrightarrow \mathbb{M}^3_m$$ 
is an isometric embedding such that  $X_r (S_r) $ is  a star-shaped, convex surface with 
\be \label{eq-Ric-nur} 
\begin{split}
\bRic  (\nu_r, \nu_r) =  - 2 m r^{-3}   \left( 1 +  O ( r^{- \tau} ) \right)   .
\end{split}
\ee
Here $ \nu_r $  is the outward unit normal to $ X_r (\mathbb{S}^2) $ in $ \mathbb{M}^3_m$
and \eqref{eq-Ric-nur} follows from \eqref{eq-Ricr-nur}.
Thus, we have proved the first part of Theorem \ref{thm-intro-iso-emb} on the existence of 
the desired isometric embedding of $(S_r, g_r)$ into $\mathbb{M}^3_m$. 

\subsection{Evaluation of the limits}
To prove the remaining part  of Theorem \ref{thm-intro-iso-emb}, 
we write $ X_r  = ( \rho_r , \theta_r  )$. By  \eqref{eq-txr}, we have 
\be \label{eq-l-Xr}
\| \rho_r - r \|_{C^{2, \alpha} } =  O (r^{1-\tau}).
\ee 
 Similarly, if  $H_m$ denotes  the mean curvature of $X_r (S_r) $ in $ \mathbb{M}^3_{m}$, 
 then \eqref{eq-txr} gives 
 \be \label{eq-Hm-r}
H_m  = {2}r^{-1} + O ( r^{- \tau - 1 } ) . 
\ee

We first compute  $\int_{S_r}NHd\sigma$,
where  $N = \left( 1 - \frac{2m}{\rho} \right)^\frac12 $ is the static potential on $\mathbb{M}^3_m$
and $ H$ is  the mean curvature of $S_r$ in $(M, \fg)$.
Let $ A(r)$ be the area of $(S_r, g_r)$.
By  \cite[Lemma 2.1]{FST}, 
\begin{align*}
H= {2}r^{-1}+O(r^{-1-\tau}) \ \ \mathrm{and} \  \  A(r)=4\pi r^2+O(r^{2-\tau}) . 
\end{align*}
By \cite[Lemma 2.2]{FST}, 
\begin{align*}
\int_{S_r}H d\sigma = r^{-1} A(r) +4\pi r-8\pi \m +o(1) .
\end{align*}
Therefore,  by \eqref{eq-l-Xr}, 
\begin{equation}  \label{limit-1}
\begin{split}
\int_{S_r}NHd \sigma = & \ 
\int_{S_r} \left(1- {m}r^{-1} \right)H d\sigma +o(1) \\
= & \ r^{-1} {A(r)} + 4\pi r-8\pi \m - 8 \pi m +o(1) . 
\end{split}
\end{equation}

Next, we compute   $\int_{S_r}NH_m d \sigma$. 
Identifying   $ S_r $ with its image  $ \Sigma_r = X_r (S_r)  $,
we carry out the computation  in  $\mathbb{M}^3_m$. 
Following notations   in Section \ref{sec-flow}, 
we  rewrite the Schwarzschild metric 
$g_m = \frac{1}{N^2 } d \rho^2 + \rho^2 \sigma $ 
 as 
\begin{align*}
g_m =d s^2+\phi^2 (s) \sigma 
\end{align*}
by setting  $ s = \int_{2m}^\rho  \frac{1}{N (t) } d t $. 
Then 
$ \phi (s) = \rho $ and $ \phi' (s) = N .$
Define 
$$\Phi (s) =\int_0^s  \phi (t) d t  \  \ \mathrm{and} \ \ 
 u=\left\langle\phi\partial_s,\nu_r \right\rangle . $$ 
On $\Sigma_r$, \eqref{eq-Hess-Phi} becomes 
\begin{align} \label{eq-conformal-K}
\Phi_{; ij}=\phi^\prime {g_r}_{ij}-h_{ij}u , 
\end{align}
where  $h$ is the second fundamental form of $ \Sigma_r $.
Taking trace of \eqref{eq-conformal-K} gives
\be \label{eq-Mink-1}
0=2\int_{\Sigma_r} \phi^\prime d\sigma -\int_{\Sigma_r} H_m u \, d \sigma . 
\ee
Now we apply  \cite[Lemma 2.5]{GLW} to get another Minkowski type identity.
Precisely, let $\sigma_2^{ij}=\frac{\partial \sigma_2}{\partial h_{ij}}=\sigma_1 {g_r}^{ij}-h^{ij}$. 
Contracting  $\sigma_2^{ij}$  with $\Phi_{ij}$ shows 
\be \label{eq-Mink-2}
\int_{\Sigma_r}\sigma_2^{ij}\Phi_{; ij}d\sigma =\int_{S_r} H_m  \phi^\prime d\sigma -2\int_{S_r} \sigma_2 u \, d \sigma . 
\ee
Integrating  by parts and applying the Codazzi equation, 
we have
\begin{align*}
\int_{S_r}\sigma_2^{ij}\Phi_{; ij} d\sigma = & \ -\int_{S_r}(\sigma_2^{ij})_{;j} \Phi_{;i} \, d\sigma 
=\int_{S_r}\bRic(\nu_r, \nabla \Phi ) \, d\sigma ,
\end{align*}
where $ \nabla \Phi$ is the gradient of $ \Phi$ on $ \Sigma_r$.  By \eqref{eq-l-Xr}, 
\be \label{eq-Phi-i}
| \nabla \Phi |^2 = {g_r}^{ij}  \Phi_{;i} \Phi_{;j} = O ( r^ {2 - 2 \tau}) .
\ee
This combined  with the fact  $ | \bRic(\nu_r, \cdot ) |  =O(r^{-3})$  shows 
\begin{align*}
\int_{S_r}\sigma_2^{ij}\Phi_{;ij}d\sigma=o(1) .
\end{align*}
Therefore, by \eqref{eq-Mink-2}, 
\be \label{eq-Mink-2-o}
\int_{S_r} H_m \phi^\prime \, d\sigma=2\int_{S_r} \sigma_2 u \, d\sigma+o(1).
\ee
Note that
$ u^2  = | \bar \nabla \Phi |^2 -   | \nabla \Phi |^2 $, 
where $ \bar \nabla $ denotes the gradient on $ \mathbb{M}^3_m$.  Thus, by \eqref{eq-Phi-i}, 
\be
u = r + O ( r^{1 - \tau } ) . 
\ee

Now  let $ K$ be the Gauss curvature of $(S_r, g_r)$. 
By \cite[Lemma 2.1]{FST}, if we let $\bar{K}=K-r^{-2}$, then $ \bar K =O(r^{-2-\tau})$. 
Thus, by  the Gauss equation and  \eqref{eq-Ric-nur}, 
\begin{align*}
\sigma_2=K+\bRic (\nu_r,\nu_r)=\bar{K}+r^{-2}-2mr^{-3}+o(r^{-3}) .
\end{align*}
Following the steps in \cite{FST}, we   have
\begin{align}\label{limit-2-1}
\int_{\Sigma_r} H_m\phi^\prime d\sigma
&=2\int_{\Sigma_r} (\bar{K}+r^{-2}) u \, d\sigma -4mr^{-3}\int_{\Sigma_r}u \, d\sigma +o(1)\\\nonumber
&=2  r^{-2} {\int_{\Sigma_r} \langle \bar \nabla \Phi , \nu_r \rangle dvol}+2\int_{S_r}\bar{K} u \, d\sigma-16\pi m+o(1)\\\nonumber
&=6 r^{-2}  {\int_{\Omega_r} \phi^\prime \dvol} +2r\int_{S_r}(K- r^{-2}) \, d\sigma -16\pi m+o(1)\\\nonumber
&=6 r^{-2}  {\int_{\Omega_r}  \phi^\prime \dvol} +8\pi r- 2 r^{-1} A(r) -16\pi m+o(1),
\end{align}
where $ \Omega_r$ is the bounded  domain enclosed by $ \Sigma_r$ and the horizon boundary 
of $ \mathbb{M}^3_m$ and $ \dvol $ is the volume element on $\mathbb{M}^3_m$. 

Next, let  $\bar{H}_m =H_m- 2 r^{-1}$. By \eqref{eq-Hm-r}, $ \bar H  =O(r^{-1-\tau})$. By \eqref{eq-Mink-1}, 
\begin{equation*}
\begin{split}
2\int_{\Sigma_r}\phi^\prime \, d\sigma &=\int_{\Sigma_r}H_m u \, d\sigma 
=\int_{\Sigma_r}( 2 r^{-1} +\bar{H}_m ) u \, d\sigma \\
&=6 r^{-1} {\int_{\Omega_r} \phi^\prime \dvol}+\int_{\Sigma_r}\bar{H}_m  u \, d\sigma .
\end{split}
\end{equation*}
Since $u=r+O(r^{1-\tau})$ and  $\phi^\prime=N = 1+O(r^{-1})$, we have $u=r\phi^\prime+O(r^{1-\tau})$. Thus,
\begin{equation} \label{eq-Hmu-1}
\begin{split}
2\int_{\Sigma_r}\phi^\prime \, d\sigma & =6 r^{-1} {\int_{\Omega_r}  \phi^\prime \dvol} 
+\int_{\Sigma_r}\bar{H}_m  (r\phi^\prime+O(r^{1-\tau})) \, d\sigma \\
&=6 r^{-1} {\int_{\Omega_r} \phi^\prime \dvol} + r\int_{\Sigma_r}H_m\phi^\prime \, d\sigma 
-2\int_{\Sigma_r}\phi^\prime \, d\sigma +O(r^{2-2\tau}) .
\end{split}
\end{equation}
Since $\phi^\prime=N = 1 - m r^{-1} +O(r^{-1-\tau})$, we also have 
\begin{equation} \label{eq-Hmu-2}
\int_{\Sigma_r}\phi^\prime \, d\sigma =A(r)-4\pi m r+ O ( r^{1- \tau} ). 
\end{equation}
Thus, it follows from \eqref{eq-Hmu-1} and \eqref{eq-Hmu-2} that 
\begin{align}\label{limit-2-2}
\int_{\Sigma_r} H_m\phi^\prime \, d\sigma  = - 6 r^{-2} \int_{\Omega_r} \phi^\prime dvol +4 r^{-1} {A(r)} -16\pi m +o(1).
\end{align}
Combining  (\ref{limit-2-1}) and (\ref{limit-2-2}), and replacing $ \phi' $ by $N$,  we have
\begin{align}\label{limit-2}
\int_{S_r} N H_m  \, d\sigma =4\pi r+\frac{A(r)}{r}-16\pi m+o(1) .
\end{align}
By  (\ref{limit-1}) and (\ref{limit-2}), we therefore  conclude
\begin{align*}
\int_{S_r}N(H_m-H) \, d\sigma =-8\pi m+8\pi \m +o(1), 
\end{align*}
or equivalently 
\begin{align*}
\lim_{r\rightarrow \infty}\left(m+\frac{1}{8\pi} \int_{S_r}N(H_m-H) \, d\sigma \right)= \m ,
\end{align*}
which proves \eqref{eq-intro-nby-limit}.

To prove \eqref{eq-intro-volume-limit}, by  (\ref{limit-2-1}) and (\ref{limit-2-2}),  we also have 
\be \label{eq-N-bulk}
 {\int_{\Omega_r}  N \,  \dvol}  =  {\int_{\Omega_r}  \phi' \,  \dvol}  
=   \frac12  r {A(r)}  -  \frac23 \pi r^3   +o(r^2) . 
\ee
Let $ V(r)$ be the volume of the region enclosed by $S_r$ in $(M, \fg)$.
By (2.28) in \cite{FST}, 
\be \label{eq-V-fst}
V(r)  =   \frac12  r {A(r)}  -  \frac23 \pi r^3   + 2 \pi \m r^2 + o(r^2)  .
\ee
Hence, it follows from  \eqref{eq-N-bulk} and \eqref{eq-V-fst} that 
\be \label{eq-N-bulk-V}
 {\int_{\Omega_r}  N \,  \dvol}   - V(r) =   - 2 \pi \m r^2 + o(r^2)   .
\ee
Next, let $V_m (r)$ denote the volume of $ \Omega_r$ in $\mathbb{M}^3_m$.
We claim
\be \label{eq-int-N-Vm}
{\int_{\Omega_r}  N \,  \dvol}  = V_m (r)  - 2 \pi m r^2 + o(r^2)  . 
\ee
To see this,  let $ D_\rho$ denote the region in $\mathbb{M}^3_m$  bounded by the rotationally symmetric sphere with area $4\pi \rho^2$ 
and the horizon boundary. 
Let $ \rho_0 > 2 m $ be a fixed constant such that, for any $ \rho > \rho_0$, 
\be \label{eq-N-est}
 \left |  N  - \left( 1 - \frac{m}{\rho} \right) \right | \le C_1 \rho^{-2} ,
\ee
where $C_1 > 0 $ is  independent on $\rho$. 
By \eqref{eq-l-Xr} and \eqref{eq-N-est}, for large $r$, we have
\be \label{eq-N-Vol-0}
\begin{split}
\int_{\Omega_r} N  dvol = & \  \int_{\Omega_r \setminus D_{\rho_0} } N  \, \dvol  +   O(1)  \\
= & \  \int_{\Omega_r \setminus D_{\rho_0} } \left(  1 - \frac{m}{\rho}  \right)  \,  \dvol  + O(r)   \\
= & \  V_m (r)  -  \int_{\Omega_r \setminus D_{\rho_0} }  \frac{m}{\rho}  \,  \dvol  +  O(r)  . 
\end{split} 
\ee
By \eqref{eq-l-Xr}, we also have
$$
\int_{ ( D_{r - C  r^{1 -\tau}  } )  \setminus D_{\rho_0} } \rho^{-1}  \, \dvol 
 \le \int_{\Omega_r \setminus D_{\rho_0} } \rho^{-1}  \,  \dvol  
\le  \int_{ ( D_{r + C r^{1 -\tau}  } )  \setminus D_{\rho_0} } \rho^{-1} \,  \dvol  ,
$$
which   implies 
\be \label{eq-rho-1}
\int_{\Omega_r \setminus D_{\rho_0} } \rho^{-1}   dvol   = 2 \pi r^2 + o (r^2)  . 
\ee
Thus,  \eqref{eq-int-N-Vm}  follows from \eqref{eq-N-Vol-0} and  \eqref{eq-rho-1}. 
By \eqref{eq-N-bulk-V}  and  \eqref{eq-int-N-Vm}, we conclude that   
\begin{equation*}
V(r) - V_m (r) = 2 \pi ( \m - m ) r^2 + o (r^2) , 
\end{equation*}
which proves \eqref{eq-intro-volume-limit} of Theorem \ref{thm-intro-iso-emb}. 

\vspace{.3cm}

We end this paper with the following corollary. 

\begin{coro}
Let $(M^{3}, \fg)$ be an asymptotically flat $3$-manifold with nonnegative scalar curvature, 
with boundary $\p M$ being an outer minimizing minimal surface  (with one or more components).  
Let $ S_r$ denote  the large coordinate  sphere in $(M^3, \fg)$ with the induced metric $g_r$.
Let  $ m = \sqrt{ \frac{ | \partial M |} { 16 \pi} }  $. 
For large $r$, let $ X_r$ be the isometric embedding of $(S_r, g_r)$ into $ \mathbb{M}^3_m $
given by Theorem \ref{thm-intro-iso-emb}. 
Let $ V(r)$ and $V_m(r)$ be the volume of the region enclosed by $S_r$ in $(M^3, \fg)$ and 
the region enclosed by $ X_r (S_r)$ in $\mathbb{M}^3_m$, respectively. 
Then  
$$
\lim_{r \to \infty} \frac{ V(r) - V_m (r) }{ 2\pi r^2} \ exists \ and \ is \ge 0 ,
$$
and ``$=$" holds if and only if $(M^3, \fg )$ is isometric to $\mathbb{M}^3_m$. 
\end{coro}

\begin{proof}
This follows directly from \eqref{eq-intro-volume-limit} and the $3$-dimensional Riemannian Penrose inequality. 
\end{proof}

\end{document}